\documentclass[12pt, oneside]{amsart}

\usepackage[utf8]{inputenc}
\usepackage[T1]{fontenc}

\usepackage{amssymb}
\usepackage{amsthm}
\usepackage{amsmath}
\usepackage[only,Yup]{stmaryrd} %%% PARA O FIBRADO DE MOLINO
\usepackage{mathrsfs}    %%% PARA PSEUDOGRUPOS

\usepackage[a4paper,left=2cm,top=3cm,right=2cm,bottom=2cm]{geometry}

\usepackage[all,cmtip]{xy}

\usepackage{enumerate}

\usepackage{indentfirst}

\usepackage{hyperref}
\hypersetup{colorlinks=true,allcolors=black}  %%% COR DO LINK

\usepackage{arydshln}

\DeclareMathOperator{\codim}{codim}
\DeclareMathOperator{\grad}{grad}

\DeclareMathOperator{\rank}{rank}
\DeclareMathOperator{\symrank}{symrank}

\DeclareMathOperator{\Zero}{Zero}
\DeclareMathOperator{\ric}{Ric}
\DeclareMathOperator{\ind}{ind}

\mathchardef\ordinarycolon\mathcode`\:
\mathcode`\:=\string"8000
\begingroup \catcode`\:=\active
  \gdef:{\mathrel{\mathop\ordinarycolon}}
\endgroup

\begin{document}

\title{Positively curved Killing foliations via deformations}

\author{Francisco C.~Caramello Jr.}
\address{Francisco C.~Caramello Jr., Departamento de Matemática, Universidade Federal de São Carlos, Rod.~Washington Luís, Km 235, 13565-905, São Carlos - SP, Brazil}
\email{franciscocaramello@dm.ufscar.br}
\thanks{The first author was supported by CAPES, Brazil.}

\author{Dirk Töben}
\address{Dirk Töben, Departamento de Matemática, Universidade Federal de São Carlos, Rod.~Washington Luís, Km 235, 13565-905, São Carlos - SP, Brazil}
\email{dirktoben@dm.ufscar.br}


%%% AMBIENTES %%%
\theoremstyle{definition}
\newtheorem{example}{Example}[section]
\newtheorem{definition}[example]{Definition}
\newtheorem{remark}[example]{Remark}
\theoremstyle{plain}
\newtheorem{proposition}[example]{Proposition}
\newtheorem{theorem}[example]{Theorem}
\newtheorem{lemma}[example]{Lemma}
\newtheorem{corollary}[example]{Corollary}
\newtheorem{claim}[example]{Claim}
\newtheorem{thmx}{Theorem}
\renewcommand{\thethmx}{\Alph{thmx}} % "letter-numbered" theorems
\newtheorem{corx}{Corollary}
\renewcommand{\thecorx}{\Alph{corx}} % "letter-numbered" corollaries

%%% MACROS %%%
\newcommand{\dif}[0]{\mathrm{d}}
\newcommand{\od}[2]{\frac{\dif #1}{\dif #2}}
\newcommand{\pd}[2]{\frac{\partial #1}{\partial #2}}
\newcommand{\dcov}[2]{\frac{\nabla #1}{\dif #2}}
\newcommand{\proin}[2]{\left\langle #1, #2 \right\rangle}
\newcommand{\f}[0]{\mathcal{F}}
\newcommand{\g}[0]{\mathcal{G}}
\newcommand\blfootnote[1]{%
  \begingroup
  \renewcommand\thefootnote{}\footnote{#1}%
  \addtocounter{footnote}{-1}%
  \endgroup
}

\begin{abstract}
We show that a compact manifold admitting a Killing foliation with positive transverse curvature fibers over finite quotients of spheres or weighted complex projective spaces, provided that the singular foliation defined by the closures of the leaves has maximal dimension. This result is obtained by deforming the foliation into a closed one while maintaining transverse geometric properties, which allows us to apply results from the Riemannian geometry of orbifolds to the space of leaves. We also show that the basic Euler characteristic is preserved by such deformations. Using this fact we prove that a Riemannian foliation of a compact manifold with finite fundamental group and nonvanishing Euler characteristic is closed. As another application we obtain that, for a positively curved Killing foliation of a compact manifold, if the structural algebra has sufficiently large dimension then the basic Euler characteristic is positive.
\end{abstract}

\maketitle
\setcounter{tocdepth}{1}
\tableofcontents

\section{Introduction}

The symmetry rank of a Riemannian manifold is the rank of its isometry group, that is, the dimension of a maximal torus in $\mathrm{Iso}(M)$. It was proved by K.~Grove and C.~Searle in \cite{grove} that a closed Riemannian manifold with positive sectional curvature and maximal symmetry rank is diffeomorphic to a sphere, a real or complex projective space or a lens space. A generalization of this result for Alexandrov spaces was obtained recently in \cite{harvey}. Here we show the following transverse analog of these results for Killing foliations, that is, complete Riemannian foliations with globally constant Molino sheaf (see Section \ref{subsection: Molino Sheaf and Killing Foliations}). This class of foliations includes Riemannian foliations on complete simply-connected manifolds and foliations induced by isometric actions on complete manifolds \cite[Lemme III]{molino3}.

\begin{thmx}\label{theorem: harvey-searle para folheações} 
Let $\f$ be a $q$-codimensional, transversely orientable Killing foliation of a compact manifold $M$. If the transverse sectional curvature $\sec_{\f}$ of $\f$ is positive, then
$$\dim(\overline{\f})-\dim(\f)\leq\left\lfloor\frac{\codim(\f)+1}{2}\right\rfloor$$
where $\overline{\f}$ denotes the singular Riemannian foliation of $M$ by the closures of the leaves of $\f$. Moreover, if equality holds above, there is a closed Riemannian foliation $\g$ of $M$ arbitrarily close to $\f$ and such that $M/\g$ is homeomorphic to either
\begin{enumerate}[(i)]
\item a quotient of a sphere $\mathbb{S}^q/\Lambda$, where $\Lambda$ is a finite subgroup of the centralizer of the maximal torus in $\mathrm{O}(q+1)$, or
\item a quotient of a weighted complex projective space $|\mathbb{CP}^{q/2}[\lambda]|/\Lambda$, where $\Lambda$ is a finite subgroup of the torus acting linearly on $\mathbb{CP}^{q/2}[\lambda]$ (this case occurs only when $q$ is even).
\end{enumerate}
\end{thmx}

Recall that a generalized Seifert fibration is a foliation $\g$ whose leaves are all closed submanifolds. In this case we also say that $\g$ is a \textit{closed} foliation. The number $d=\dim(\overline{\f})-\dim(\f)$, where $\dim(\overline{\f})$ is the maximal dimension of the leaf closures of $\f$, is called the \textit{defect} of $\f$ and equals the dimension of the structural algebra $\mathfrak{g}_\f$ of $\f$ (see Section \ref{subsection: Molino Theory}). When $\f$ is a Killing foliation this number is also a lower bound for the symmetry rank of $\f$, that is, the maximal number of linearly independent, commuting transverse Killing vector fields that $\f$ admits. Theorem \ref{theorem: harvey-searle para folheações} has a direct application for Riemannian foliations on positively curved manifolds (see Corollary \ref{Corollary to H-S for foliations}).

We prove Theorem \ref{theorem: harvey-searle para folheações} by reducing it to application of the orbifold version of the aforementioned Grove--Searle classification \cite[Corollary E]{harvey}, after deforming $\f$ into a closed Riemannian foliation $\g$ while maintaining the relevant transverse geometric properties, using the result below. In what follows $\mathcal{T}^*(\f)$ denotes the algebra of $\f$-basic tensor fields (see Section \ref{section: basic cohomology}).

\begin{thmx}\label{theorem: Haefliger deformation}
Let $(\f,\mathrm{g}^T)$ be a Killing foliation of a compact manifold $M$. There is a homotopic deformation $\f_t$ of $\f=\f_0$ into a closed foliation $\g=\f_1$ such that
\begin{enumerate}[(i)]
\item the deformation occurs within the closures of the leaves of $\f$ and $\g$ can be chosen arbitrarily close to $\f$,
\item for each $t$ there is an injection $\iota_t:\mathcal{T}^*(\f)\to\mathcal{T}^*(\f_t)$ such that $\iota_t(\xi)$ is a smooth time-dependent tensor field on $M$, for any $\xi\in\mathcal{T}^*(\f)$, and $\iota_t(\mathrm{g}^T)$ is a transverse metric for $\f_t$,
\item the quotient orbifold $M/\!/\g$ admits a smooth effective action of a torus $\mathbb{T}^d$, where $d$ is the defect of $\f$, which is isometric with respect to the metric induced by $\iota_t(\mathrm{g}^T)$ and such that $M/\overline{\f}\cong(M/\g)/\mathbb{T}^d$.
\end{enumerate}
In particular, $\symrank(\g)\geq d$ and, if $\sec_{\f}> 0$ and $\g$ sufficiently close to $\f$, then $\sec_{\g}> 0$.
\end{thmx}

The fact that $\iota_t(\mathrm{g}^T)$ is a transverse metric for $\f_t$ is actually a particular case of a stronger property of the deformation. In fact, $\iota(\xi)$ essentially consists of a deformation of the kernel of $\xi$, so for any transverse geometric structure defined by a tensor field $\sigma\in\mathcal{T}^*(\f)$ (such as a transverse Riemannian metric, a basic orientation form, a transverse symplectic structure etc), $\iota_t(\sigma)$ will be a transverse geometric structure of the same kind for $\f_t$. In particular, $\iota_t$ restricts to an injection $\Omega_B^*(\f)\to\Omega_B^*(\f_t)$ between the algebras of basic differential forms. It is also worth to mention that, if $\pi:M\to M/\!/\g$ denotes the canonical projection, $\pi_*\circ\iota$ is an isomorphism between $\mathcal{T}^*(\f)$ and the subalgebra $\mathcal{T}^*(M/\!/\g)^{\mathbb{T}^d}$ of $\mathbb{T}^d$-invariant tensor fields on $M/\!/\g$ that, similarly, relates $\f$-transverse geometric structures to $\mathbb{T}^d$-invariant geometric structures on $M/\!/\g$.

The proof of Theorem \ref{theorem: Haefliger deformation} is based on a result by A.~Haefliger and E.~Salem \cite{haefliger2} that expresses $\f$ as the pullback of a homogeneous foliation of an orbifold. As Theorem \ref{theorem: harvey-searle para folheações} exemplifies, with this technique we are able to reduce results about the transverse geometry of Killing foliations to theorems about Riemannian orbifolds. Another application is the following ``closed leaf'' result, that complements a theorem by G.~Oshikiri \cite[Proposition and Theorem 2]{oshikiri} by allowing noncompact ambient manifolds. We show it by reducing to an application of the orbifold version of the Synge--Weinstein Theorem.

\begin{thmx}\label{theorem: Berger for foliations}
Let $\f$ be an even-codimensional, complete Riemannian foliation of a manifold $M$ with $|\pi_1(M)|<\infty$. If $\sec_{\f}\geq c>0$ then $\f$ possesses a closed leaf.
\end{thmx}

For any smooth foliation $(M,\f)$ one can define the cohomology of the differential subcomplex $\Omega_B^*(\f)\subset \Omega^*(M)$, the basic cohomology of $\f$ (see Section \ref{section: basic cohomology}). We study the basic Euler characteristic $\chi_B(\f)$, the alternate sum of the dimensions of the basic cohomology groups, when $\f$ is a Killing foliation. We show that in this case this invariant localizes to the set $\Sigma^{\dim\f}$ that consists of the union of the closed leaves of $\f$.

\begin{thmx}\label{theorem: basic euler char localizes to closed leaves}
If $\f$ is a Killing foliation of a compact manifold $M$, then $\chi_B(\f)=\chi(\Sigma^{\dim\f}/\f)$.
\end{thmx}

Equivalently, using the language of transverse actions (see Section \ref{subsection: Molino Sheaf and Killing Foliations}), $\f|_{\Sigma^{\dim\f}}$ coincides with $\f^{\mathfrak{g}_\f}$, the fixed-point set of the action of the structural algebra $\mathfrak{g}_\f$ on $\f$ (see Section \ref{subsection: Molino Theory}), so the formula in Theorem \ref{theorem: basic euler char localizes to closed leaves} becomes $\chi_B(\f)=\chi_B(\f^{\mathfrak{g}_\f})$, in analogy with the localization of the classical Euler characteristic to the fixed-point set of a torus action.

Theorem \ref{theorem: basic euler char localizes to closed leaves} enables us to show that the basic Euler characteristic is preserved by the deformations devised in Theorem \ref{theorem: Haefliger deformation} (see Theorem \ref{theorem: basic euler char is preserved by deformations}). Using this fact we obtain the following transverse analog of the partial answer to Hopf's conjecture by T.~Püttmann and C.~Searle for manifolds with large symmetry rank \cite[Theorem 2]{puttmann}, by reducing it to an orbifold version of this same result that we also prove (see Corollary \ref{corollary: Putmann for orbifolds}).

\begin{thmx}\label{theorem: puttmann for foliations}
Let $\f$ be an even-codimensional, transversely orientable Killing foliation of a closed manifold $M$. If $\sec_\f>0$ and 
$$\dim(\overline{\f})-\dim(\f)\geq\frac{\codim(\f)}{4}-1,$$
then $\chi_B(\f)>0$.
\end{thmx}

Finally, we combine the results of A.~Haefliger on the classifying spaces of holonomy groupoids that appear in \cite{haefliger3} with the invariance of $\chi_B$ by deformations to obtain the following topological obstruction for Riemannian foliations.

\begin{thmx}\label{teo: topological obstruction}
Any Riemannian foliation of a compact manifold $M$ with $|\pi_1(M)|<\infty$ and $\chi(M)\neq0$ is closed.
\end{thmx}

It is shown in \cite[Théorème 3.5]{ghys} that a Riemannian foliation of a compact simply-connected manifold $M$ satisfying $\chi(M)\neq 0$ must have a closed leaf, a fact that also follows from the Poincaré--Hopf index theorem (see Proposition \ref{prop: X killing com Zero=Sigma dim f}). Notice that Theorem \ref{teo: topological obstruction} shows, in fact, that all leaves of such a foliation must be closed.

\section{Preliminaries}\label{section: preliminaires}

In this section we review some of the basics on Riemannian foliations and establish our notation. All objects are supposed to be of class $C^\infty$ (smooth) unless otherwise stated.

Throughout this article, $\f$ will denote a $p$-dimensional foliation of a \emph{connected} manifold $M$ without boundary, of dimension $n=p+q$. The subbundle of $TM$ consisting of the spaces tangent to the leaves will be denoted by $T\f$ and the Lie algebra of the vector fields with values in $T\f$ by $\mathfrak{X}(\f)$. The number $q$ is the \textit{codimension} of $\f$. A foliation $\f$ is \textit{transversely orientable} if its normal space $\nu\f:=TM/T\f$ is orientable. The set of the closures of the leaves in $\f$ is denoted by $\overline{\f}:=\{\overline{L}\ |\ L\in\f\}$. In the simple case when $\overline{\f}=\f$, we say that $\f$ is a \textit{closed} foliation.

Recall that $\f$ can be alternatively defined by an open cover $\{U_i\}_{i\in I}$ of $M$, submersions $\pi_i:U_i\to \overline{U}_i$, with $\overline{U}_i\subset\mathbb{R}^q$, and diffeomorphisms $\gamma_{ij}:\pi_j(U_i\cap U_j)\to\pi_i(U_i\cap U_j)$ satisfying $\gamma_{ij}\circ\pi_j|_{U_i\cap U_j}=\pi_i|_{U_i\cap U_j}$ for all $i,j\in I$. The collection $\{\gamma_{ij}\}$ is a \textit{Haefliger cocycle} representing $\f$ and each $U_i$ is a \textit{simple open set} for $\f$. We will assume without loss of generality that the fibers $\pi_i^{-1}(\overline{x})$ are connected.

A \textit{foliate morphism} between $(M,\f)$ and $(N,\g)$ is a smooth map $f:M\to N$ that sends leaves of $\f$ into leaves of $\g$. In particular, we may consider $\f$-foliate diffeomorphisms $f:M\to M$. The infinitesimal counterparts of this notion are the \textit{foliate vector fields} of $\f$, that is, the vector fields in the Lie subalgebra
$$\mathfrak{L}(\f)=\{X\in\mathfrak{X}(M)\ |\ [X,\mathfrak{X}(\f)]\subset\mathfrak{X}(\f)\}.$$
If $X\in\mathfrak{L}(\f)$ and $\pi:U\to \overline{U}$ is a submersion locally defining $\f$, then $X|_U$ is $\pi$-related to some vector field $\overline{X}_{\overline{U}}\in\mathfrak{X}(\overline{U})$. In fact, this characterizes $\mathfrak{L}(\f)$ \cite[Section 2.2]{molino}. The Lie algebra $\mathfrak{L}(\f)$ also has the structure of a module, whose coefficient ring consists of the \textit{basic functions} of $\f$, that is, functions $f\in C^{\infty}(M)$ such that $Xf=0$ for every $X\in\mathfrak{X}(\f)$. We denote this ring by $\Omega^0_B(\f)$. The quotient of $\mathfrak{L}(\f)$ by the ideal $\mathfrak{X}(\f)$ yields the Lie algebra $\mathfrak{l}(\f)$ of \textit{transverse vector fields}. For $X\in\mathfrak{L}(\f)$, we denote its induced transverse field by $\overline{X}\in\mathfrak{l}(\f)$. Note that each $\overline{X}$ defines a unique section of $\nu\f$ and that $\mathfrak{l}(\f)$ is also a $\Omega^0_B(\f)$-module.

Let $(M,\f)$ be represented by the Haefliger cocycle $\{(U_i,\pi_i,\gamma_{ij})\}$. The pseudogroup of local diffeomorphisms generated by $\gamma=\{\gamma_{ij}\}$ acting on $T_\gamma:=\bigsqcup_i \overline{U}_i$ is the \textit{holonomy pseudogroup} of $\f$ associated to $\gamma$, that we denote by $\mathscr{H}_\gamma$. If $\delta$ is another Haefliger cocycle defining $\f$ then $\mathscr{H}_\delta$ is equivalent to $\mathscr{H}_\gamma$, meaning that there is a maximal collection $\Phi$ of diffeomorphisms $\varphi$ from open sets of $T_\delta$ to open sets of $T_\gamma$ such that $\{\mathrm{Dom}(\varphi)\ |\ \varphi\in\Phi\}$ covers $T_\delta$, $\{\mathrm{Im}(\varphi)\ |\ \varphi\in\Phi\}$ covers $T_\gamma$ and, for all $\varphi,\psi\in\Phi$, $h\in\mathscr{H}_\delta$ and $h'\in\mathscr{H}_\gamma$, we have $\psi^{-1}\circ h'\circ\varphi\in\mathscr{H}_\delta$, $\psi\circ h\circ\varphi^{-1}\in\mathscr{H}_\gamma$ and $h'\circ\varphi\circ h\in\Phi$.

We will write $(T_\f,\mathscr{H}_\f)$ to denote a representative of the equivalence class of these pseudogroups. Note that the orbit space $T_\f/\mathscr{H}_\f$ is precisely $M/\f$. We denote the germinal holonomy group of a leaf $L$ at $x\in L$ by $\mathrm{Hol}_x(L)$. Recall that there is a surjective homomorphism $h:\pi_1(L,x)\to\mathrm{Hol}_x(L)$. Also, it can be shown that leaves without holonomy are generic, in the sense that $\{x\in M\ |\ \mathrm{Hol}_x(L)=0\}$ is residual in $M$ \cite[Chapter 2]{candel}.

\subsection{Foliations of orbifolds}\label{section: orbifolds}

Orbifolds are generalizations of manifolds that arise naturally in many areas of mathematics. We refer to \cite[Chapter 1]{adem}, \cite[Section 2.4]{mrcun} or  \cite{kleiner} for quick introductions.

Let $\mathcal{O}$ be a smooth orbifold with an atlas $\mathcal{A}=\{(\widetilde{U}_i,H_i,\phi_i)\}_{i\in I}$. We denote the underlying space of $\mathcal{O}$ by $|\mathcal{O}|$. Recall that each chart of $\mathcal{A}$ consists of an open subset $\widetilde{U}$ of $\mathbb{R}^n$, a finite subgroup $H$ of $\mathrm{Diff}(\widetilde{U})$ and an $H$-invariant map $\phi:\widetilde{U}\to |\mathcal{O}|$ that induces a homeomorphism between $\widetilde{U}/H$ and some open subset $U\subset |\mathcal{O}|$. Given a chart $(\widetilde{U},H,\phi)$, and $x=\phi(\tilde{x})\in U$, the \textit{local group} $\Gamma_x$ of $\mathcal{O}$ at $x$ is the isotropy subgroup $H_{\tilde{x}}<H$. Its isomorphism class is independent of both the chart and the lifting $\tilde{x}$, and for every $x\in|\mathcal{O}|$ we can always find a compatible chart $(\widetilde{U}_x,\Gamma_x,\phi_x)$ \textit{around} $x$, that is, such that $\phi^{-1}(x)$ consists of a single point.

\begin{example} If $\f$ is a $q$-codimensional foliation of a manifold $M$ such that every leaf is compact and with finite holonomy, then $M/\f$ has a canonical $q$-dimensional orbifold structure relative to which the local group of a leaf in $M/\f$ is its holonomy group \cite[Theorem 2.15]{mrcun}. We will denote the orbifold obtained this way by $M/\!/\f$, in order to distinguish it from the topological space $M/\f$. Analogously, if a Lie group $G$ acts properly and almost freely on $M$, we will denote the quotient orbifold by $M/\!/G$ (see also Example \ref{exe: foliated actions}).
\end{example}

Similarly to the construction of the holonomy pseudogroup of a foliation, consider $U_\mathcal{A}:=\bigsqcup_{i\in I}\widetilde{U}_i$ and $\phi:=\{\phi_i\}_{i\in I}:U_\mathcal{A}\to |\mathcal{O}|$, that is, $x\in \widetilde{U}_i\subset U_\mathcal{A}$ implies $\phi(x)=\phi_i(x)$. A \textit{change of charts} of $\mathcal{A}$ is a diffeomorphism $h:V\to W$, with $V,W\subset U_\mathcal{A}$ open sets, such that $\phi\circ h=\phi|_V$. The collection of all changes of charts of $\mathcal{A}$ form a pseudogroup $\mathscr{H}_{\mathcal{A}}$ of local diffeomorphisms of $U_\mathcal{A}$, and $\phi$ induces a homeomorphism $U_\mathcal{A}/\mathscr{H}_{\mathcal{A}}\to|\mathcal{O}|$. The germs of the elements in $\mathscr{H}_{\mathcal{A}}$ have a natural structure of a Lie groupoid that furnishes an alternative framework to develop the theory of orbifolds (see \cite[Chapter 5]{mrcun} or \cite[Section 1.4]{adem}).

A smooth differential form on $\mathcal{O}$ can be seen as an $\mathscr{H}_{\mathcal{A}}$-invariant differential form on $U_\mathcal{A}$. The De Rham cohomology of $\mathcal{O}$ is defined in complete analogy with the manifold case. The following result can be seen as an orbifold version of De Rham's Theorem (see \cite[Theorem 1]{satake} or \cite[Theorem 2.13]{adem}).

\begin{theorem}[Satake]\label{theorem: Satake}
Let $\mathcal{O}$ be an orbifold. Then $H_{\mathrm{dR}}^i(\mathcal{O})\cong H^i(|\mathcal{O}|,\mathbb{R})$.
\end{theorem}

There are several distinct definitions of smooth maps between orbifolds in the literature (see, for instance, \cite{borzellino2}, \cite{chenruan} and \cite{kleiner}), the most common being to define a continuous map $f:|\mathcal{O}|\to|\mathcal{P}|$ to be smooth when, for every $x\in|\mathcal{O}|$, there are charts $(\widetilde{U},H,\phi)$ around $x$ and $(\widetilde{V},K,\psi)$ around $f(x)$ such that $f(U)\subset V$ and there is a smooth map $\widetilde{f}:\widetilde{U}\to\widetilde{V}$ with $\psi\circ\widetilde{f}=f\circ\phi$. There are relevant refinements of this notion, such as the \textit{good maps} defined in \cite{chenruan}, that are needed in some constructions, for example in order to coherently pull orbibundles back by smooth maps. We remark here that this notion of good map matches the notion of smooth morphisms when the orbifolds are viewed as Lie groupoids (see \cite[Proposition 5.1.7]{lupercio}). In particular, a smooth map $M\to\mathcal{O}$ ``in the orbifold sense'', as defined in \cite[p. 715]{haefliger2}, corresponds to a good map.

Following \cite[Section 3.2]{haefliger2}, we define a smooth foliation $\f$ of $\mathcal{O}$ to be a smooth foliation of $U_\mathcal{A}$ which is invariant by $\mathscr{H}_{\mathcal{A}}$. The atlas can be chosen so that on each $\widetilde{U}_i$ the foliation is given by a surjective submersion with connected fibers onto a manifold $T_i$.

\subsection{Basic cohomology}\label{section: basic cohomology}

A covariant tensor field $\xi$ on $M$ is $\f$-basic if $\xi(X_1,\dots,X_k)=0$, whenever some $X_i\in\mathfrak{X}(\f)$, and $\mathcal{L}_X\xi=0$ for all $X\in\mathfrak{X}(\f)$. These are the tensor fields that project to tensor fields on $T_\f$ and are invariant by $\mathscr{H}_\f$. In particular, we say that a differential form $\alpha\in\Omega^i(M)$ of degree $i$ is \textit{basic} when it is basic as a tensor field. By Cartan's formula, $\alpha$ is basic if, and only if, $i_X\alpha=0$ and $i_X(d\alpha)=0$ for all $X\in\mathfrak{X}(\f)$. We denote the $\Omega^0_B(\f)$-module of basic $i$-forms of $\f$ by $\Omega^i_B(\f)$.

By definition, $\Omega^*_B(\f)$ is closed under the exterior derivative, so we can consider the complex
$$\cdots \stackrel{d}{\longrightarrow} \Omega^{i-1}_B(\f) \stackrel{d}{\longrightarrow} \Omega^i_B(\f) \stackrel{d}{\longrightarrow} \Omega^{i+1}_B(\f) \stackrel{d}{\longrightarrow}\cdots .$$
The cohomology groups of this complex are the \textit{basic cohomology groups} of $\f$, that we denote by $H^i_B(\f)$. When the dimensions $\dim(H^i_B(\f))$ are all finite we define the \textit{basic Euler characteristic} of $\f$ as the alternate sum
$$\chi_B(\f)=\sum_i(-1)^i\dim(H^i_B(\f)).$$
In analogy with the manifold case, we say that $b_B^i(\f):=\dim(H^i_B(\f))$ are the \textit{basic Betti numbers} of $\f$. When $\f$ is the trivial foliation by points, we recover the classical Euler characteristic and Betti numbers of $M$.

Since we have an identification between $\f$-basic forms and $\mathscr{H}_\f$-invariant forms on $T_\f$, and an identification between differential forms on an orbifold $\mathcal{O}$ and $\mathscr{H}_{\mathcal{O}}$-invariant forms on $U_\mathcal{O}$, the following result is clear.

\begin{proposition}\label{prop: basic cohomology of closed foliations}
Let $(M,\f)$ be a smooth foliation whose every leaf is compact and with finite holonomy. Then the projection $\pi:M\to M/\!/\f$ induces an isomorphism of differential complexes $\pi^*:\Omega^*(M/\!/\f)\to\Omega_B^*(\f)$. In particular, $H_B^*(\f)\cong H_{\mathrm{dR}}^*(M/\!/\f)$.
\end{proposition}

\subsection{Riemannian foliations}

Let $\f$ be a smooth foliation of $M$. A \textit{transverse metric} for $\f$ is a symmetric, positive, basic $(2,0)$-tensor field $\mathrm{g}^T$ on $M$. In this case $(M,\f,\mathrm{g}^T)$ is called a \textit{Riemannian foliation}. A Riemannian metric $\mathrm{g}$ on $M$ is \textit{bundle-like} for $\f$ if for any open set $U$ and any vector fields $Y,Z\in\mathfrak{L}(\f|_U)$ that are perpendicular to the leaves, $\mathrm{g}(Y,Z)\in\Omega^0_B(\f|_U)$. Any bundle-like metric $\g$ determines a transverse metric by $\mathrm{g}^T(X,Y):=\mathrm{g}(X^\bot,Y^\bot)$ with respect to the decomposition $TM=T\f\oplus T\f^\perp$. Conversely, given $\mathrm{g}^T$ one can always choose a bundle-like metric on $M$ that induces it \cite[Proposition 3.3]{molino}. With a chosen bundle-like metric, we will identify the bundles $\nu\f\equiv T\f^\perp$.

A metric $\mathrm{g}$ is bundle-like for $(M,\f)$ if and only if a geodesic that is perpendicular to a leaf at one point remains perpendicular to all the leaves it intersects. Moreover, geodesic segments perpendicular to the leaves project to geodesic segments on the local quotients $\overline{U}$. It follows that the leaves of a Riemannian foliation are locally equidistant.

\begin{example}\label{exe: foliated actions}
If a foliation $\f$ on $M$ is given by an almost free action of a Lie group $G$ and $\mathrm{g}$ is a Riemannian metric on $M$ such that $G$ acts by isometries, then $\mathrm{g}$ is bundle-like for $\f$ \cite[Remark 2.7(8)]{mrcun}. In other words, a foliation induced by an isometric action is Riemannian. A foliation given by the action of a Lie group is said to be \textit{homogeneous}.

For a specific example, consider the flat torus $\mathbb{T}^2=\mathbb{R}^2/\mathbb{Z}^2$. For each $\lambda\in(0,+\infty)$ consider the almost free $\mathbb{R}$-action $(t,[x,y]) \longmapsto [x+t,y+\lambda t]$. The resulting foliation is the \textit{$\lambda$-Kronecker foliation} of the torus. Observe that when $\lambda$ is irrational each leaf is dense in $\mathbb{T}^2$.

A more relevant example is the following. Let us fix $\lambda=(\lambda_0,\dots,\lambda_n)\in\mathbb{N}^{n+1}$ satisfying $\gcd(\lambda_0,\dots,\lambda_n)=1$ and modify the standard action of $\mathbb{C}^\times$ on $\mathbb{C}^{n+1}\setminus\{0\}$ by adding weights given by $\lambda$. Precisely, let $z\in\mathbb{C}^\times$ act by $z\!\cdot\!(z_0,\dots,z_n)=(z^{\lambda_0}z_0,\dots,z^{\lambda_n}z_n)$. This is an almost free isometric action that restricts to an action of $\mathbb{S}^1<\mathbb{C}^\times$ on $\mathbb{S}^{2n+1}\subset\mathbb{C}^{n+1}$ with the same quotient. The quotient orbifold $\mathbb{S}^{2n+1}/\!/\mathbb{S}^1$ is called a \textit{weighted projective space} and denoted $\mathbb{CP}^n[\lambda]$ (for further details, see \cite[Section 4.5]{boyer}). Note that $\mathbb{CP}^1[p,q]$, for example, is simply the $\mathbb{Z}_p$-$\mathbb{Z}_q$-football orbifold, that is, a sphere with two cone singularities (of order $p$ and $q$) at the poles.
\end{example}

\begin{example}\label{example: Riemannian suspensions}
Let $(T,\mathrm{g})$ be a Riemannian manifold. A foliation $\f$ defined by the suspension of a homomorphism $h:\pi_1(B,x_0)\to\mathrm{Iso}(T)$ is naturally a Riemannian foliation \cite[Section 3.7]{molino}.
\end{example}

It follows from the definition that $\mathrm{g}^T$ projects to Riemannian metrics on the local quotients $\overline{U}_i$ of a Haefliger cocycle $\{(U_i,\pi_i,\gamma_{ij})\}$ defining $\f$. The holonomy pseudogroup $\mathscr{H}_\f$ then becomes a pseudogroup of local isometries of $T_\f$ and, with respect to a a bundle-like metric, the submersions defining $\f$ become Riemannian submersions. We will say that $\f$ has positive \textit{transverse sectional curvature} when $(T_\f,\mathrm{g}^T)$ has positive sectional curvature. In this case we denote $\sec_\f>0$. The notions of negative, nonpositive and nonnegative transverse sectional curvature are defined analogously, as well as the corresponding notions for \textit{transverse Ricci curvature}.

The basic cohomology of Riemannian foliations on compact manifolds have finite dimension, as shown in \cite[Théorème 0]{kacimi2}. The hypothesis that $M$ is compact can be relaxed to $M/\overline{\f}$ being compact, provided that $\f$ is a \textit{complete Riemannian foliation}, that is, that $M$ is a complete Riemannian manifold with respect to some bundle-like metric for $\f$ \cite[Proposition 3.11]{goertsches}. Hence, if this is the case, $\chi_B(\f)$ is always defined.

The following transverse analogue of the Bonnet--Myers Theorem due to J.~Hebda will be useful \cite[Theorem 1]{hebda}.

\begin{theorem}[Hebda]\label{theorem: hebda}
Let $\f$ be a complete Riemannian foliation satisfying $\ric_\f\geq c>0$. Then $M/\f$ is compact and $H^1_B(\f)=0$.
\end{theorem}

\subsection{Molino theory}\label{subsection: Molino Theory}

In this section we summarize the structure theory for Riemannian foliations due mainly to P.~Molino \cite{molino}. Recall that a transverse metric induces a Riemannian metric on the quotient of each foliation chart $U$. The pullbacks of the Levi-Civita connections on $\overline{U}$ glue together to a well-defined connection $\nabla^B$ on $TM$, the \textit{canonical basic Riemannian connection}, which induces a covariant derivative on $\mathfrak{l}(\f|_U)$ (that we also denote by $\nabla^B$).

Let $\pi^\Yup:M^\Yup_\f\to M$ be the principal $\mathrm{O}(q)$-bundle of $\f$-transverse orthonormal frames\footnote{When $\f$ is transversely orientable $M^\Yup_\f$ consists of two $\mathrm{SO}(q)$-invariant connected components that correspond to the possible orientations. In this case we will assume that one component was chosen and, by abuse of notation, denote it also by $M^\Yup_\f$. Everything stated in this section then will carry over to this case by changing $\mathrm{O}(q)$ to $\mathrm{SO}(q)$.}. The normal bundle $\nu\f$ is associated to $M^\Yup_\f$, so the basic Riemannian connection $\nabla^B$ on $\nu\f$ induces a connection form $\omega_{\f}$ on $M^\Yup_\f$. This connection form in turn defines a horizontal distribution $\mathcal{H}:=\ker(\omega_{\f})$ on $M^\Yup_\f$ that allows us to horizontally lift $\f$, obtaining a foliation $\f^\Yup$ of $M^\Yup_\f$.

The advantage of lifting $\f$ to $\f^\Yup$ is that the latter admits a global transverse parallelism, that is, $\nu\f^\Yup$ is parallelizable by fields in $\mathfrak{l}(\f^\Yup)$ \cite[p. 82, p.148]{molino}. If $\f$ is complete, those fields admit complete representatives in $\mathfrak{L}(\f^\Yup)$ \cite[Section 4.1]{goertsches}, so $\f$ is \textit{transversally complete}, in the terminology of \cite[Section]{molino}. The theory of transversely parallelizable foliations then states that the partition $\overline{\f^\Yup}$ of $M^\Yup_\f$ is a \textit{simple foliation}, that is, $W:=M^\Yup_\f/\overline{\f^\Yup}$ is a manifold and $\overline{\f^\Yup}$ is given by the fibers of a locally trivial fibration $b:M^\Yup_\f\to W$ \cite[Proposition 4.1']{molino}. Molino shows that a leaf closure $\overline{L}\in\overline{\f}$ is the image by $\pi^\Yup$ of a leaf closure of $\f^\Yup$, which implies that each leaf closure is an embedded submanifold of $M$ \cite[Lemma 5.1]{molino}.

Let us fix $L^\Yup\in\f^\Yup$, denote $J=\overline{L^\Yup}$, and consider the foliation $(J,\f^\Yup|_J)$ and the Lie algebra $\mathfrak{g}_{\f}:=\mathfrak{l}(\f^\Yup|_J)$. Then $\f^\Yup|_J$ is a complete Lie $\mathfrak{g}_{\f}$-foliation, in the terminology of E.~Fedida, whose work establishes that such foliations are \textit{developable}, meaning that they lift to simple foliations of the universal coverings \cite[Theorem 4.1]{molino}. The restriction of $\f^\Yup$ to the closure of a different leaf is isomorphic to $(J,\f^\Yup|_J)$, so $\mathfrak{g}_{\f}$ is an algebraic invariant of $\f$, called its \textit{structural algebra}. We say that $d:=\dim(\mathfrak{g}_{\f})$ is the \textit{defect} of $\f$, motivated by the results in the next section.

\subsection{Molino sheaf and Killing foliations}\label{subsection: Molino Sheaf and Killing Foliations}

A field $X\in\mathfrak{X}(M)$ is a \textit{Killing vector field for $\mathrm{g}^T$} if $\mathcal{L}_X\mathrm{g}^T=0$. These fields form a Lie subalgebra of $\mathfrak{L}(\f)$ \cite[Lemma 3.5]{molino} and there is, thus, a corresponding Lie algebra of \textit{transverse Killing vector fields} that we  denote by $\mathfrak{iso}(\f,\mathrm{g}^T)$ (we will omit the transverse metric when it is clear from the context). The elements of $\mathfrak{iso}(\f,\mathrm{g}^T)$ are precisely the transverse fields that project to Killing vector fields on the local quotients of $\f$.

Now suppose that $\f$ is a complete Riemannian foliation and consider on $M^\Yup_\f$ the sheaf of Lie algebras $\mathscr{C}_{\f^\Yup}$ that associates to an open set $U^\Yup\subset M^\Yup_\f$ the Lie algebra $\mathscr{C}_{\f^\Yup}(U^\Yup)$ of the $\f^\Yup$-transverse fields in $U^\Yup$ that commute with all the global fields in $\mathfrak{l}(\f^\Yup)$. Each field in $\mathscr{C}_{\f^\Yup}(U^\Yup)$ is the natural lift of a $\f$-transverse Killing vector field on $\pi^\Yup(U^\Yup)$ \cite[Proposition 3.4]{molino}. The push-forward $\pi^\Yup_*(\mathscr{C}_{\f^\Yup})$ will be called the \textit{Molino sheaf} of $\f$, that we denote simply by  $\mathscr{C}_{\f}$. From what we just saw, it is the sheaf of the Lie algebras consisting of the local transverse Killing vector fields that lift to local sections of $\mathscr{C}_{\f^\Yup}$.

The sheaf $\mathscr{C}_{\f}$ is Hausdorff \cite[Lemma 4.6]{molino} and its stalk on each point is isomorphic to the Lie algebra $\mathfrak{g}_{\f}^{-}$ opposite to $\mathfrak{g}_{\f}$ \cite[Proposition 4.4]{molino}. The main motivation for the study of $\mathscr{C}_{\f}$ is that its orbits are the closures of the leaves of $\f$ \cite[Theorem 5.2]{molino}, in the sense that
$$\{X_x\ |\ X\in(\mathscr{C}_{\f})_x\}\oplus T_xL_x=T_x\overline{L_x}.$$
Notice that the relationship between $\overline{\f}$, $\mathscr{C}_{\f}$ and $\mathfrak{g}_\f$ enables us to write the defect of $\f$ as
$$d=\dim(\mathfrak{g}_\f)=\dim(\overline{\f})-\dim(\f)=\codim(\f)-\codim(\overline{\f}).$$

An interesting class of Riemannian foliations is the one consisting of complete Riemannian foliations that have a \textit{globally} constant Molino sheaf. Such foliations are called \textit{Killing foliations}, following \cite{mozgawa}. In other words, if $\f$ is a Killing foliation then there exists $\overline{X}_1,\dots,\overline{X}_d\in\mathfrak{iso}(\f)$ such that $T\overline{\f}=T\f\oplus\langle X_1,\dots, X_d \rangle$. A complete Riemannian foliation $\f$ is a Killing foliation if and only if $\mathscr{C}_{\f^\Yup}$ is globally constant, and in this case $\mathscr{C}_{\f^\Yup}(M^\Yup_\f)$ is the center of $\mathfrak{l}(\f^\Yup)$. Hence $\mathscr{C}_{\f}(M)$ is central (but not necessarily the full center of) $\mathfrak{l}(\f)$. It follows that the structural algebra of a Killing foliation is Abelian, because we have $\mathfrak{g}_{\f}^{-}\cong\mathscr{C}_{\f}(M)\cong(\mathscr{C}_{\f})_x$ for each $x\in M$. In particular, $\mathfrak{g}_{\f}^{-}\cong \mathfrak{g}_{\f}$ if $\f$ is Killing.

A complete Riemannian foliation $\f$ of a simply-connected manifold is automatically a Killing foliation \cite[Proposition 5.5]{molino}, since in this case $\mathscr{C}_{\f}$ cannot have holonomy. Homogeneous Riemannian foliations provide another relevant class of Killing foliations.

\begin{example}\label{example: homogeneous foliations are killing}
If $\f$ is a Riemannian foliation of a complete manifold $M$ given by the connected components of the orbits of a locally free action of $H<\mathrm{Iso}(M)$, then $\f$ is a Killing foliation and $\mathscr{C}_{\f}(M)$ consists of the transverse Killing vector fields induced by the action of $\overline{H}\subset\mathrm{Iso}(M)$ \cite[Lemme III]{molino3}.
\end{example}

In the terminology of transverse actions introduced in \cite[Section 2]{goertsches}, if $\f$ is a Killing foliation we can simply say that there is an effective isometric transverse action of $\mathfrak{g}_{\f}$ on $(M,\f)$, given by the Lie isomorphism $\mathfrak{g}_{\f}\ni V\mapsto \overline{V}^*\in\mathscr{C}_{\f}(M)<\mathfrak{iso}(\f)$, such that the singular foliation everywhere tangent to the distribution of varying rank $\mathfrak{g}_{\f}\!\cdot\!\f$ defined by $(\mathfrak{g}_{\f}\!\cdot\!\f)_x:=\{\overline{V}^*_x\ |\ V\in\mathfrak{g}_{\f}\}\oplus T_x\f$ is the singular foliation of $M$ given by the closures of the leaves of $\f$ \cite[Theorem 2.2]{goertsches}. For short, we write this as $\mathfrak{g}_{\f}\!\cdot\!\f=\overline{\f}$.

\section{Transverse Killing vector fields and the structural algebra}

In this section we prove some basic facts about the zero sets of transverse Killing vector fields that will later be used. We also study the behavior of the structural algebra when a Riemannian foliation is lifted to a finitely-sheeted covering space.

\subsection{Properties of transverse Killing vector fields}\label{section: Elementary Properties of Transverse Killing Fields}

If $(M,\f,\mathrm{g}^T)$ is a Riemannian foliation and $\overline{X}\in\mathfrak{iso}(\f,\mathrm{g}^T)$ is a transverse Killing vector field, we denote by $\Zero(\overline{X})$ the set where $\overline{X}$ vanishes. We say that an $\f$-saturated submanifold $N\subset M$ is \textit{horizontally totally geodesic} if it projects to totally geodesic submanifolds in the local quotients $\overline{U}$ of $\f$.

\begin{proposition}\label{proposition: propriedade zero killing}
Let $(M,\f,\mathrm{g}^T)$ be a Riemannian foliation and let $\overline{X}\in\mathfrak{iso}(\f)$ be a transverse Killing vector field. Then each connected component $N$ of $\Zero(\overline{X})$ is a closed submanifold of $M$ of even codimension. Moreover, $N$ is horizontally totally geodesic and saturated by the leaves of $\f$ and $\overline{\f}$, and if $\f$ is transversely orientable then $(N,\f|_N)$ is transversely orientable.
\end{proposition}

\begin{proof}
Since $\overline{X}$ is transverse and $\Zero(\overline{X})$ is closed, $N$ is saturated by the leaves of both $\f$ and $\overline{\f}$. Recall that, for a local trivialization $\pi:U\to \overline{U}$ of $\f$, $\overline{X}$ projects to a Killing vector field $\overline{X}_{\overline{U}}$ on $(\overline{U},\pi_*(\mathrm{g}^T))$, so we see that $N\cap U=\pi^{-1}(\overline{N})$, where $\overline{N}$ is a connected component of $\Zero(\overline{X}_{\overline{U}})$, which in turn is known to be a totally geodesic submanifold of $\overline{U}$ of even codimension \cite[Theorem 5.3]{kobayashi}. Using this it is now easy to prove that $N$ is an even-codimensional, horizontally totally geodesic, closed submanifold of $M$.

For the last claim, note that $(\nabla^B\overline{X})_x: \nu_x\f\to\nu_x\f$ is skew-symmetric and vanishes on $T_xN/T_xL_x$, so it preserves the $\mathrm{g}^T$-orthogonal complement $W_x:=(T_xN/T_xL_x)^\perp$. Choosing an adequate orthonormal basis, $(\nabla^B\overline{X})_x|_{W_x}$ has a matrix representation of the type
$$\begin{bmatrix}
0 & \alpha_1 & & & \\
-\alpha_1& 0 & & & \\
 & & \ddots &  &  \\
& & & 0 & \alpha_l \\
& & & -\alpha_l & 0\end{bmatrix},$$
where $l=\dim(W_x)/2=\codim(N)/2$. The eigenvalues $\pm i\alpha_j$ remain constant on $N$ because on each simple open set $U\ni x$ these are the eigenvalues of $(\nabla^{\overline{U}}(\overline{X}_{\overline{U}}))_{\overline{x}}$, that are known to be constant on $\overline{N}$ \cite[p. 61]{kobayashi}. The bundle $W=(TN/T\f)^\perp$ therefore decomposes into $\nabla^B\overline{X}$-invariant subbundles $W=E_1\oplus\dots\oplus E_l$ and we see that $J|_{E_j}=1/\alpha_j\nabla^B\overline{X}$ defines a complex structure on $W$. In particular, we have that $W$ is orientable and, since
$$\left.\frac{TM}{T\f}\right|_N=\frac{TN}{T\f}\oplus W,$$
the result follows.
\end{proof}

Recall that the symmetry rank of a Riemannian manifold is the rank of its isometry group. In analogy, if $(M,\f,\mathrm{g}^T)$ is a Riemannian foliation we define the \textit{transverse symmetry rank} of $\f$ by
$$\symrank(\f):=\max_{\mathfrak{a}}\Big\{\dim(\mathfrak{a})\Big\},$$
where $\mathfrak{a}$ runs over all the Abelian subalgebras of $\mathfrak{iso}(\f)$. If $\mathfrak{a}<\mathfrak{iso}(\f)$ is a subalgebra satisfying $\dim(\mathfrak{a})=\symrank(\f)$, we will denote by $\mathcal{Z}(\mathfrak{a})$ the set consisting of all proper connected components of the zero sets of the transverse Killing vector fields in $\mathfrak{a}$.

\begin{proposition}\label{proposition: propriedades de Zero killing}
For any $N,N'\in\mathcal{Z}(\mathfrak{a})$ the following holds:
\begin{enumerate}[(i)]
\item Every transverse Killing vector field in $\mathfrak{a}$ is tangent to $N$ (that is, every foliate field representing a field in $\mathfrak{a}$ is tangent to $N$) and, therefore, the restriction to $N$ of the fields in $\mathfrak{a}$ yields a commutative Lie algebra $\mathfrak{a}|_N$ of transverse Killing vector fields of $\f|_N$.
\item If $N$ is maximal in $\mathcal{Z}(\mathfrak{a})$ with respect to set inclusion, then
$$\dim(\mathfrak{a}|_N)=\dim(\mathfrak{a})-1.$$
\item Each connected component of $N\cap N'$ also belongs to $\mathcal{Z}(\mathfrak{a})$.
\end{enumerate}
\end{proposition}

The proof of the analogous properties for Killing vector fields on a Riemannian manifold $M$ \cite[Proposition 30]{petersen} adapts directly to the present setting if one works on the normal spaces $\nu_x\f$ in place of $T_xM$ and with the basic connection $\nabla^B$, so we will omit the details.

\subsection{The canonical stratification}

Given a complete Riemannian foliation $\f$ of $M$, we define $\dim(\overline{\f})=\max_{L\in\f}\{\dim(\overline{L})\}$. For $s$ satisfying $\dim(\f)\leq s\leq\dim(\overline{\f})$, let us denote by $\Sigma^s$ the subset of points $x\in M$ such that $\dim(\overline{L_x})=s$. Then we get a decomposition
$$M=\bigsqcup_x \Sigma_x,$$
called the \textit{canonical stratification} of $\f$, where $\Sigma_x$ is the connected component of $\Sigma^s$ that contains $x$. Each component $\Sigma_x$ is an embedded submanifold \cite[Lemma 5.3]{molino} called a \textit{stratum} of $\f$. The restriction $\overline{\f}|_{\Sigma_x}$ now has constant dimension and forms a (regular) Riemannian foliation \cite[Lemma 5.3]{molino}. The \textit{regular stratum} $\Sigma^{\dim\overline{\f}}$ is an open, connected and dense subset of $M$, and each other stratum $\Sigma_x\neq \Sigma^{\dim\overline{\f}}$ is called \textit{singular} and satisfies $\codim(\Sigma_x)\geq2$ \cite[p. 197]{molino}. The subset $\Sigma^{\dim\f}$ will be called the \textit{stratum of the closed leaves}, even though it is not, in general, a canonical stratum.

\begin{proposition}\label{prop: X killing com Zero=Sigma dim f}
Let $(M,\f)$ be a Killing foliation. There exists a transverse Killing vector field $\overline{X}\in\mathfrak{iso}(\f)$ such that $\Zero(\overline{X})=\Sigma^{\dim\f}$.
\end{proposition}

\begin{proof}
Choose an (at most) countable cover $\{U_i\}$ of $M\setminus\Sigma^{\dim\f}$ by simple open sets. Since there are no closed leaves in $U_i$, the algebra $\mathscr{C}_\f(U_i)=\mathscr{C}_\f(M)|_{U_i}$ projects on the quotient $\overline{U}_i$ to an Abelian algebra $\mathfrak{c}_i$ of Killing vector fields whose orbits have dimension at least $1$. It is then known that the set of Killing vector fields in $\mathfrak{c}_i$ that do not vanish at any point of $\overline{U}_i$ is residual \cite[Lemme]{mozgawa}. Clearly, $\overline{X}\in\mathscr{C}_\f(M)$ vanishes at $x\in U_i$ if, and only if, the induced Killing vector field $\overline{X}_i\in\mathfrak{c}_i$ vanishes at $\overline{x}$. Hence, since we have only countable many open sets $U_i$, it follows that the set of fields in $\mathscr{C}_\f(M)$ not vanishing at any point of $M\setminus\Sigma^{\dim\f}$ is residual in $\mathscr{C}_\f(M)$. In other words, a generic $\overline{X}\in\mathscr{C}_\f(M)<\mathfrak{iso}(\f)$ satisfies $\Zero(\overline{X})=\Sigma^{\dim\f}$.
\end{proof}

In particular, each connected component $N$ of $\Sigma^{\dim\f}$ is a horizontally totally geodesic, closed submanifold of $M$ of even codimension, and $\f|_N$ is transversely orientable when $\f$ is (see Proposition \ref{proposition: propriedade zero killing}).

\subsection{The structural algebra and finite coverings}

Let us now study the behavior of the structural algebra when a Riemannian foliation is lifted to a finitely-sheeted covering space.

\begin{lemma}\label{lemma: levantamentos finitos}
Let $\f$ be a smooth foliation of a smooth manifold $M$ and let $\rho:\widehat{M}\to M$ be a finitely-sheeted covering map. For any $\widehat{L}\in\widehat{\f}$, we have $\overline{\rho(\widehat{L})}=\rho(\overline{\widehat{L}})$, where $\widehat{\f}=\rho^*(\f)$ is the lifted foliation of $\widehat{M}$. In particular, $\rho:\overline{\widehat{L}}\to\overline{\rho(\widehat{L})}$ is a finitely-sheeted covering.
\end{lemma}

The proof is elementary, so we will omit it.

\begin{proposition}\label{proposition: Molino sheaf under liftings}
Let $\f$ be a complete Riemannian foliation of $M$ and let $\rho:\widehat{M}\to M$ be a finitely-sheeted covering map. Then $\mathscr{C}_{\widehat{\f}}=\rho^*(\mathscr{C}_\f)$, where $\widehat{\f}=\rho^*(\f)$.
\end{proposition}

\begin{proof}
We can identify $\widehat{M}^\Yup_{\widehat{\f}}$ with the pullback bundle $(\rho^\Yup)^*(M^\Yup_\f)$, so we have a commutative diagram
$$\xymatrix{
\widehat{M}^\Yup_{\widehat{\f}} \ar[r]^{\rho^\Yup} \ar[d]_{\hat{\pi}} & M^\Yup_\f \ar[d]^{\pi}\\
\widehat{M} \ar[r]^\rho & M}$$
where the horizontal arrows are finitely-sheeted covering maps. Moreover, $(\rho^\Yup)^*(\mathscr{C}_{\f^\Yup})$ can be identified with the lift $\hat{\pi}^*(\rho^*(\mathscr{C}_\f))$, hence, since the Molino sheaves of $\f$ and $\widehat{\f}$ are defined in terms of the sheaves of the corresponding lifted foliations, it remains to show that $(\rho^\Yup)^*(\mathscr{C}_{\f^\Yup})=\mathscr{C}_{\widehat{\f}^\Yup}$.

Indeed, $\mathscr{C}_{\widehat{\f}^\Yup}$ commutes with the Lie algebra $\mathfrak{l}(\widehat{\f}^\Yup)$ of $\widehat{\f}^\Yup$-transverse fields, hence, in particular, it commutes with $(\rho^\Yup)^*(\mathfrak{l}(\f^\Yup))$, so $\mathscr{C}_{\widehat{\f}^\Yup}$ is a subsheaf of $(\rho^\Yup)^*(\mathscr{C}_{\f^\Yup})$. This implies that if we consider an open subset $U\subset \widehat{M}^\Yup_{\widehat{\f}}$ where both $\mathscr{C}_{\widehat{\f}^\Yup}$ and $(\rho^\Yup)^*(\mathscr{C}_{\f^\Yup})$ are constant and such that $\rho^\Yup|_U$ is a diffeomorphism, then $\rho_*^\Yup(\mathscr{C}_{\widehat{\f}^\Yup}(U))<\mathscr{C}_{\f^\Yup}(\rho^\Yup(U))$. Now, by Lemma \ref{lemma: levantamentos finitos}, the leaf closures in $\widehat{\f}^\Yup$ and $\f^\Yup$ have the same dimension, hence we must have
$$\rho_*^\Yup(\mathscr{C}_{\widehat{\f}^\Yup}(U))=\mathscr{C}_{\f^\Yup}(\rho^\Yup(U)),$$
therefore $\mathscr{C}_{\widehat{\f}^\Yup}(U)=(\rho^\Yup)^*(\mathscr{C}_{\f^\Yup})(U)$ for any small enough $U$.\end{proof}

\begin{corollary}\label{corollary: algebra estrutural e levantamentos finitos}
Let $\f$ be a complete Riemannian foliation of $M$ and let $\rho:\widehat{M}\to M$ be a finitely-sheeted covering. Then $\mathfrak{g}_\f\cong\mathfrak{g}_{\widehat{\f}}$, where $\widehat{\f}$ is the lifted foliation of $\widehat{M}$. In particular, if $|\pi_1(M)|<\infty$, then $\mathfrak{g}_\f$ is Abelian.
\end{corollary}

\section{Deformations of Killing foliations}

There are several notions of deformations of foliations available in the literature \cite[Section 3.6]{candel}. Here we will be interested in deformations of the following type: two smooth foliations $\f_0$ and $\f_1$ of a manifold $M$ are $C^\infty$-\textit{homotopic} if there is a smooth foliation $\f$ of $M\times [0,1]$ such that $M\times\{t\}$ is saturated by leaves of $\f$, for each $t\in[0,1]$, and
$$\f_i=\f|_{M\times\{i\}},$$
for $i=0,1$. In this case we will also say that $\f$ is a \textit{homotopic deformation} of $\f_0$ into $\f_1$

\begin{example}
Consider, for $\lambda_1,\lambda_2\in\mathbb{R}$, the $\lambda_i$-Kronecker foliations $\f(\lambda_i)$ of $\mathbb{T}^2$ (see Example \ref{exe: foliated actions}). Clearly, a homotopic deformation between $\f(\lambda_1)$ and $\f(\lambda_2)$ is given by $\f((1-t)\lambda_1+t\lambda_2)$, $t\in [0,1]$. Notice that if $\lambda_1$ is irrational and we choose $\lambda_2\in\mathbb{Q}$, then we obtain a deformation of a foliation with dense leaves into a closed foliation. We  are primarily interested in deformations with this property.
\end{example}

\subsection{Proof of Theorem \ref{theorem: Haefliger deformation}}\label{Section: proof of the deformation theorem}

To deform any Killing foliation $\f$ with compact leaf closures into a closed foliation, whilst maintaining some properties of its transverse geometry, we will use the following theorem by A.~Haefliger and E.~Salem \cite[Theorem 3.4]{haefliger2}.

\begin{theorem}[Haefliger--Salem]\label{theorem: Haefliger-Salem}
There is a bijection between the set of equivalence classes of Killing foliations $\f$ with compact leaf closures of $M$ and the set of equivalence classes of quadruples $(\mathcal{O},\mathbb{T}^N,H,\mu)$, where $\mathcal{O}$ is an orbifold, $\mu$ is an action of $\mathbb{T}^N$ on $\mathcal{O}$ and $H$ is a dense contractible subgroup of $\mathbb{T}^N$ whose action on $\mathcal{O}$ is locally free\footnote{In this paragraph, two foliations $\f$ and $\mathcal{F'}$ are \textit{equivalent} if their holonomy pseudogroups are equivalent and two quadruples $(\mathcal{O},\mathbb{T}^N,H,\mu)$ and $(\mathcal{O}',\mathbb{T}^{N'},H',\mu')$ are \textit{equivalent} if there is an isomorphism between $\mathbb{T}^N$ and $\mathbb{T}^{N'}$ and a diffeomorphism of $\mathcal{O}$ on $\mathcal{O}'$ that conjugates the actions $\mu$ and $\mu'$}. This bijection associates $\f$ to a canonical realization $(\mathcal{O},\f_H)$ of the classifying space of its holonomy pseudogroup, where $\f_H$ is the foliation of $\mathcal{O}$ determined by the orbits of $H$. In particular, there is a smooth good map $\Upsilon:M\to\mathcal{O}$ such that $\Upsilon^*(\f_H)=\f$.
\end{theorem}

Suppose that $M$ is compact and $\f$ is a Killing foliation of $M$. We can deform $\f$ using this result as follows. Let $\mathfrak{h}$ be the Lie algebra of $H$ and consider a Lie subalgebra $\mathfrak{k}<\mathrm{Lie}(\mathbb{T}^N)\cong\mathbb{R}^N$, with $\dim(\mathfrak{k})=\dim(\mathfrak{h})$, such that its corresponding Lie subgroup $K<\mathbb{T}^N$ is closed. We can suppose $\mathfrak{k}$ close enough to $\mathfrak{h}$, as points in the Grassmannian $\mathrm{Gr}^{\dim\mathfrak{h}}(\mathrm{Lie}(\mathbb{T}^N))$, so that the action $\mu|_K$ remains locally free. Because $K$ is a closed subgroup, the leaves of the foliation $\g_K$ defined by the orbits of $K$ are all closed. Taking $\mathfrak{k}$ even closer to $\mathfrak{h}$ if necessary, we can suppose that $\Upsilon$ remains transverse to $\g_K$, so $\g:=\Upsilon^*(\g_K)$ is the desired approximation of $\f$. In this sense, $\f$ can be arbitrarily approximated by such a closed foliation $\g$. Moreover, this allows us to suppose that a submanifold $T\subset M$ that is a total transversal for $\f$ is also a total transversal for $\g$, so that there are representatives of the holonomy pseudogroups $\mathscr{H}_\f$ and $\mathscr{H}_\g$ both acting on $T$. The same goes for $\mathscr{H}_{\f_H}$ and $\mathscr{H}_{\g_K}$, they both act on the transversal $S=\Upsilon(T)$.

Theorem \ref{theorem: Haefliger-Salem} states that $(\mathcal{O},\f_H)$ is a realization of the classifying space of $\mathscr{H}_\f$. Roughly speaking, this classifying space is a space with a foliation such that the holonomy covering of each leaf is contractible and whose holonomy pseudogroup is equivalent to $\mathscr{H}_\f$. In particular, $(T,\mathscr{H}_\f)$ is equivalent to $(S,\mathscr{H}_{\f_H})$, the equivalence being generated by restrictions of $\Upsilon$ to open sets in $T$ where it becomes a diffeomorphism. One can prove that also $(T,\mathscr{H}_\g)\cong(S,\mathscr{H}_{\g_K})$, but $(\mathcal{O},\g_K)$ is not a realization of the classifying space of $\mathscr{H}_\g$, because a generic leaf is no longer contractible. It follows that $\Upsilon^*$ defines isomorphisms $\mathcal{T}^*(\f_H)\to \mathcal{T}^*(\f)$ and $\mathcal{T}^*(\g_K)\to \mathcal{T}^*(\g)$.

Now choose a $\mathbb{T}^N$-invariant normal bundle $\nu\f_H\subset T\mathcal{O}$. This can be done, for instance, by choosing a Riemannian metric on $\mathcal{O}$ with respect to which $\mathbb{T}^N$ acts by isometries. Let $\xi\in\mathcal{T}^*(\f)$ and let $\xi_H$ be the corresponding $\f_H$-basic tensor field on $\mathcal{O}$, that is, $\xi_H$ is $H$-invariant (hence $\mathbb{T}^N$-invariant), and $\xi_H(X_1,\dots,X_k)=0$ whenever some $X_i\in T\f_H$. Define a tensor field $\xi_K$ on $\mathcal{O}$ by declaring that $\xi_K=\xi_H$ on $\nu\f_H$ and that $\xi_K(X_1,\dots,X_k)=0$ whenever some $X_i\in T\g_K$. Then from what we saw above it follows that $\xi_K$ is $K$-invariant. Since $\xi_K$ vanishes in $T\g_K$ by construction, it is therefore $\g_K$-basic. The association $\xi\mapsto \Upsilon^*(\xi_K)$ defines the desired injection $\iota:\mathcal{T}^*(\f)\hookrightarrow \mathcal{T}^*(\g)$. By construction it is clear that $\iota(\mathrm{g}^T)$ is a transverse Riemannian metric for $\g$. In fact, for any $\f$-transverse structure given by an $\f$-basic tensor field, $\iota$ will induce a $\g$-transverse structure of the same kind.

Since $M$ is complete, $\mathscr{H}_\f$ is a complete pseudogroup of local isometries (in the sense of \cite[Definition 2.1]{salem}), with respect to $\mathrm{g}^T$ (see \cite[Proposition 2.6]{salem}). By \cite[Proposition 2.3]{salem}, its closure $\overline{\mathscr{H}_\f}$ in the $C^1$-topology is also a complete pseudogroup of local isometries, whose orbits are the closures of the orbits of $\mathscr{H}_\f$, which in turn correspond to leaf closures in $\overline{\f}$. The same goes for the orbits of $\overline{\mathscr{H}_{\f_H}}$: they correspond to the closures of the leaves of $\f_H$, that is, the orbits of $\mathbb{T}^N$. Hence, by the equivalence $\mathscr{H}_\f\cong\mathscr{H}_{\f_H}$, we have $M/\overline{\f}\cong \mathcal{O}/\mathbb{T}^N$. It is clear that the $\mathbb{T}^N$-action projects to an action of the torus $\mathbb{T}^N/K\cong \mathbb{T}^{N-\dim\mathfrak{k}}$ on the orbifold $\mathcal{O}/\!/\g_K$. Therefore
$$\frac{M}{\overline{\f}}\cong \frac{|\mathcal{O}|}{\mathbb{T}^N}\cong \frac{|\mathcal{O}|/\g_K}{\mathbb{T}^{N-\dim\mathfrak{k}}}.$$
Observe that, since $N=\dim(\overline{\f_H})$, we have $N-\dim(\mathfrak{k})=N-\dim(\f_H)=\codim(\f)-\codim(\overline{\f})=d$, where $d=\dim(\mathfrak{g}_\f)$ is the defect of $\f$. On the other hand, $\mathcal{O}/\!/\g_K\cong M/\!/\g$, so we get a smooth $\mathbb{T}^d$ action on $M/\!/\g$ satisfying
$$\frac{M}{\overline{\f}}\cong\frac{M/\g}{\mathbb{T}^d}.$$

If we denote the canonical projection $M\to M/\!/\g$ by $\pi^\g$, then $\pi^\g_*\circ\iota$ defines an isomorphism between $\mathcal{T}^*(\f)$ and $\mathcal{T}^*(M/\!/\g)^{\mathbb{T}^d}$. In fact, for this it suffices to see that $\xi_H\mapsto \pi^{\g_K}_*(\xi_K)$ is an isomorphism $\mathcal{T}^*(\f_H)\to\mathcal{T}^*(\mathcal{O}/\!/\g_K)^{\mathbb{T}^d}$, but since $\xi_H\mapsto\xi_K$ defines an isomorphism between $\mathcal{T}^*(\f_H)$ and $\mathcal{T}^*(\g_K)^{\mathbb{T}^N}$, this is clear. In particular, the $\mathbb{T}^d$-action on $M/\!/\g$ is isometric with respect to $\pi^\g_*\circ\iota(\mathrm{g}^T)$.

Of course, we could further consider a smooth path $\mathfrak{h}(t)$, for $t\in[0,1]$, on the Grassmannian $\mathrm{Gr}^{\dim\mathfrak{h}}(\mathrm{Lie}(\mathbb{T}^N))$ connecting $\mathfrak{h}$ to $\mathfrak{k}$ such that the action $\mu|_{H(t)}$ of each corresponding Lie subgroup $H(t)$ is locally free and the induced foliation is transverse to $\Upsilon$. Then $\f_t:=\Upsilon^*(\f_{H(t)})$ defines a $C^{\infty}$-homotopic deformation of $\f$ into $\g$. In this case we have an injection $\iota_t:\mathcal{T}^*(\f)\hookrightarrow \mathcal{T}^*(\f_t)$ for each $t$. It is clear that $\iota_t(\xi)$ is a smooth time-dependent tensor field on $M$ for each $\xi\in\mathcal{T}^*(\f)$, that is, $\iota(\xi)$ is smooth as a map $[0,1]\times M\to \bigotimes^*TM$. Since both the deformation $\f_t$ and $\iota_t(\mathrm{g}^T)$ depend smoothly on $t$, the transverse sectional curvature $\sec_{\f_t}$ with respect to $\iota_t(\mathrm{g}^T)$ is a smooth function on $t$. In particular, if $\sec_\f> c$, then since $M$ is compact we actually have $\sec_\f\geq c'> c$, hence if $\g$ is sufficiently close to $\f$ then $\sec_\g>c$, by continuity. Of course, similarly we can choose $\g$ satisfying $\sec_\g< c$ when $\sec_\f< c$.\qed
\vspace{8pt}

We will take advantage of the notation established above to observe one more fact that will be useful later. Choose a sequence $\g_i=\Upsilon^*(\g_{K_i})$ of closed foliations approaching $\f$. As the deformation respects $\overline{\f}$, if $\f$ has a closed leaf $L$, then $L\in\g_i$ for each $i$. Let us denote by $h(\g_i)$ the order of the holonomy group of $L$ as a leaf of $\g_i$. We claim that
$$\lim_{i\to\infty}h(\g_i)=\infty.$$
In fact, if $L_{\mathcal{O}}=Hx$ is the closed orbit in $\mathcal{O}$ corresponding to $L$, then $h(\g_i)\geq|(K_i)_x|$, since the $\g_{K_i}$-holonomy group of $L_{\mathcal{O}}$ is an extension of $(K_i)_x$ by the local group $\Gamma_x$ of $\mathcal{O}$. On the other hand, the stabilizer $\mathbb{T}^N_x$ is transverse to $H$ in $\mathbb{T}^N$ and, as $H$ is dense, $H_x=\mathbb{T}^N_x\cap H$ is infinite, hence it is clear that if $K_i$ approaches $H$, that is, if $\mathfrak{k}_i\to\mathfrak{h}$, then $|\mathbb{T}^N_x\cap K_i|\to \infty$. Let us state this below.

\begin{lemma}\label{lemma: variando a holonomia}
Let $\f$ be a nonclosed Killing foliation of a compact manifold $M$. If $\g_i$ is a sequence of closed foliations approaching $\f$, given by Theorem \ref{theorem: Haefliger deformation}, and $\f$ has a closed leaf $L$, then $h(\g_i)\to\infty$, where $h(\g_i)$ is the order of the holonomy group of $L$ as a leaf of $\g_i$.
\end{lemma}

\section{A closed leaf theorem}

Let us proceed with the proof of Theorem \ref{theorem: Berger for foliations}. Suppose that $\f$ has no closed leaves and consider the lift $\widehat{\f}$ of $\f$ to the universal covering space $\rho:\widehat{M}\to M$ of $M$. Then, endowed with $\rho^*(\mathrm{g}^T)$, $\widehat{\f}$ is a Killing foliation also satisfying $\sec_{\widehat{\f}}\geq c>0$. By Proposition \ref{proposition: Molino sheaf under liftings}, $\widehat{\f}$ does not have any closed leaves either. By Theorem \ref{theorem: hebda} it follows that $\widehat{M}/\overline{\widehat{\f}}$ is compact, therefore the holonomy pseudogroup $(\widehat{T},\mathscr{H}_{\widehat{\f}})$ of $\widehat{\f}$ is a complete pseudogroup of local isometries \cite[Proposition 2.6]{salem} whose space of orbit closures is compact (as it coincides with $\widehat{M}/\overline{\widehat{\f}}$). Moreover, since there is a surjective homomorphism
$$\pi_1\big(\widehat{M}\big)\to\pi_1\big(\mathscr{H}_{\widehat{\f}}\big)$$
(see \cite[Section 1.11]{salem}), we have that $\mathscr{H}_{\widehat{\f}}$ is $1$-connected. We can now apply \cite[Theorem 3.7]{haefliger2} to conclude that there exists a Riemannian foliation $\f'$ of a simply-connected compact manifold $M'$ with $(T',\mathscr{H}_{\f'})$ equivalent to $(\widehat{T},\mathscr{H}_{\widehat{\f}})$. In particular, $\f'$ is Killing and also satisfies $\sec_{\f'}\geq c>0$, since it is endowed with the transverse metric $(\mathrm{g}^T)'$ on $T'$ induced by $\rho^*(\mathrm{g}^T)$ via the equivalence. Furthermore, $\f'$ has no closed leaves, otherwise $\mathscr{H}_{\f'}$ would have a closed orbit, contradicting $\mathscr{H}_{\f'}\cong\mathscr{H}_{\widehat{\f}}$.

We now apply Theorem \ref{theorem: Haefliger deformation} to $(M',\f')$, deforming it into an even-codimensional closed Riemannian foliation $\g'$ with $\sec_{\g'}>0$. Since $M'$ is simply-connected we have that $\g'$ is transversely orientable, hence $M'/\!/\g'$ is orientable. Fixing an orientation for it, we have an even-dimensional, compact, oriented Riemannian orbifold $M'/\!/\g'$ with positive sectional curvature and an isometric action of a torus $\mathbb{T}^d$, where $d=\dim(\mathfrak{g}_{\f'})>0$. This action has no fixed points, since $M'/\overline{\f'}=(M'/\g')/\mathbb{T}^d$ and $\f'$ has no closed leaves. It is possible to choose a $1$-parameter subgroup $\mathbb{S}^1<\mathbb{T}^d$ such that $(M'/\!/\g')^{\mathbb{S}^1}=(M'/\!/\g')^T$ (see \cite[Lemma 4.2.1]{allday}), that is, without fixed points. In particular, this gives us an isometry of $M'/\!/\g'$ that has no fixed points, which contradicts the Synge--Weinstein Theorem for orbifolds (see \cite[Theorem 2.3.5]{yeroshkin}). This finishes the proof.

\section{Positively curved foliations with maximal-dimensional closure}

A generalization of the Grove--Searle classification of positively curved manifolds with maximal symmetry rank for Alexandrov spaces was obtained recently in \cite{harvey}. Since the underlying space of a positively curved orbifold is naturally a positively curved Alexandrov space, this result also furnishes a classification for orbifolds with maximal symmetry rank \cite[Corollary E]{harvey}. Theorem \ref{theorem: harvey-searle para folheações} now follows easily by combining this classification with Theorem \ref{theorem: Haefliger deformation}.

\subsection{Proof of Theorem \ref{theorem: harvey-searle para folheações}}\label{section: proof of havery-searle}

Let $\g$ be a closed foliation, given by Theorem \ref{theorem: Haefliger deformation}, satisfying $\sec_{\g}>0$. Then the leaf space of $\g$ is a $q$-dimensional Riemannian orbifold $M/\!/\g$ whose sectional curvature is positive. By the Bonnet--Myers Theorem for orbifolds \cite[Corollary 21]{borzellino3} it follows that $M/\!/\g$ is compact. Moreover, if $\omega$ is a basic orientation form for $\f$ then $\iota(\omega)$ is an orientation form for $\g$, so $M/\!/\g$ is orientable. In particular, $M/\!/\g$ has no $1$-codimensional strata. The union of the closures of the $1$-codimensional strata coincides with the Alexandrov boundary of $M/\g$, so it follows that $M/\!/\g$ is closed in the sense of \cite{harvey}.

Furthermore, since $M/\!/\g$ admits an isometric action of the torus $\mathbb{T}^d$, where $d=\dim(\overline{\f})-\dim(\f)$ is the defect of $\f$, it follows directly from \cite[Corollary E]{harvey} that $d\leq\lfloor(q+1)/2\rfloor$ and, in case of equality, that $M/\g$ is homeomorphic to one of the listed spaces.\qed
\vspace{8pt}

\begin{corollary}\label{Corollary to H-S for foliations}
Let $(M,\f)$ be a Riemannian foliation. If $\sec_M\geq c>0$ with respect to a bundle-like metric for $\f$, then
$$\dim(\overline{\f})-\dim(\f)\leq\left\lfloor\frac{\codim(\f)+1}{2}\right\rfloor$$
and, if equality holds, then the universal cover $\widehat{M}$ of $M$ fibers over $\mathbb{S}^n$ or $|\mathbb{CP}^{n/2}[\lambda]|$, meaning that it admits a closed foliation $\g$ such that $\widehat{M}/\g$ is homeomorphic to one of those spaces.
\end{corollary}

\begin{proof}
Let $(\widehat{M},\widehat{\f})$ be the lifted foliation of the universal covering space of $M$. Since $\sec_M\geq c>0$ it follows from the Bonnet--Myers Theorem that $M$ is compact and $|\pi_1(M)|<\infty$. Hence, using Proposition \ref{proposition: Molino sheaf under liftings} and Theorem \ref{theorem: harvey-searle para folheações}, we obtain that
$$\dim(\overline{\f})-\dim(\f)=\dim(\overline{\widehat{\f}})-\dim(\widehat{\f})\leq\left\lfloor\frac{\codim(\f)+1}{2}\right\rfloor=\left\lfloor\frac{\codim(\widehat{\f})+1}{2}\right\rfloor,$$
since $\widehat{\f}$ is a Killing foliation and $\sec_{\widehat{\f}}\geq\sec_{\widehat{M}}$, by O'Neill's equation.

In case of equality, let $\g$ be the closed foliation of $\widehat{M}$ given by Theorem \ref{theorem: harvey-searle para folheações}. The fact that $\widehat{M}$ is simply-connected guarantees that $\pi_1^{\mathrm{orb}}(\widehat{M}/\!/\g):=\pi_1(B(\widehat{M}/\!/\g))=0$ (see Section \ref{section: topological obstruction}), thus excluding the possibility that $\widehat{M}/\!/\g$ is a quotient of $\mathbb{S}^n$ or $\mathbb{CP}^{n/2}[\lambda]$ by a (non trivial) finite group $\Lambda$.
\end{proof}

\subsection{Isolated closed leaves}

In \cite[Proposition 8.1]{goertsches} it is shown that if a Killing foliation $\f$ of codimension $q$ admits a closed leaf, then $2d\leq q$, where $d=\dim(\mathfrak{g}_\f)$. When $q$ is even this is in line with what we obtained in Theorem \ref{theorem: harvey-searle para folheações} under the hypothesis $\sec_\f\geq c>0$ (see also Theorem \ref{theorem: Berger for foliations}). If $q$ is odd our defect bound becomes $\lfloor(q+1)/2\rfloor=(q+1)/2$, so if there exists such an odd-codimensional $\f$ satisfying $2d=q+1$ (and the hypotheses in Theorem \ref{theorem: harvey-searle para folheações}), it cannot have closed leaves. Moreover, \cite[Proposition 8.1]{goertsches} also states that if there is an \emph{isolated} closed leaf, then $q$ is even. When $2d=q$ we have the following partial converse.

\begin{proposition}\label{proposition: dim g max then isolated closed leaves}
Let $\f$ a Killing foliation of a closed manifold $M$. If $\codim(\f)=2\codim(\overline{\f})$, then $\f$ has at most finitely many (hence isolated) closed leaves.
\end{proposition}

\begin{proof}
Let $d$ be the defect of $\f$ and denote $q=\codim(\f)$, so $\codim(\f)=2\codim(\overline{\f})$ becomes $q=2d$. If $\f$ has no closed leaves there is nothing to prove, so suppose that $\Sigma^{\dim\f}$ is nonempty. To prove that the closed leaves must be isolated, consider a closed foliation $\g$ of $M$ given by Theorem \ref{theorem: Haefliger deformation}. Note that, in view of Proposition \ref{proposition: Molino sheaf under liftings}, we may suppose without loss of generality that $\f$ and $\g$ are transversely oriented. Then, $\mathcal{O}:=M/\!/\g$ is a $q$-dimensional, closed, oriented orbifold admitting an effective $\mathbb{T}^d$-action such that $M/\overline{\f}\cong |\mathcal{O}|/\mathbb{T}^d$. In particular, the fixed-point set $|\mathcal{O}|^{\mathbb{T}^d}$ is nonempty, so $(\mathcal{O},\mathbb{T}^d)$ is a torus orbifold, in the sense of \cite[Definition 3.1]{galazgarcia} and, therefore, $|\mathcal{O}|^{\mathbb{T}^d}$ consists of finitely many isolated points, by \cite[Lemma 3.3]{galazgarcia}. Since the closed leaves of $\f$ correspond to the inverse image of the points in $|\mathcal{O}|^{\mathbb{T}^d}$ by the quotient projection, the result follows.
\end{proof}

\begin{corollary}
Let $\f$ be a Riemannian foliation of a closed manifold $M$ satisfying $|\pi_1(M)|<\infty$. If $\codim(\f)=2\codim(\overline{\f})$, then $\f$ has only finitely many closed leaves.
\end{corollary}

\begin{proof}
The lift $\widehat{\f}$ of $\f$ to the universal covering space of $M$ is a Killing foliation and satisfies $\codim(\widehat{\f})=\codim(\f)=2\codim(\overline{\f})=2\codim(\overline{\widehat{\f}})$, by Corollary \ref{corollary: algebra estrutural e levantamentos finitos}. Therefore $\widehat{\f}$ (and, consequently, $\f$) have finitely many closed leaves.
\end{proof}

\section{Localization of the basic Euler characteristic}

We will prove in this section that the basic Euler characteristic (see Section \ref{section: basic cohomology}) of a Killing foliation $(M,\f)$ localizes to the zero set of any transverse Killing vector field in $\mathscr{C}_{\f}(M)$ and hence, in particular, to $\Sigma^{\dim \f}$. A useful tool will be a transverse version of Hopf's Index Theorem that appears in \cite[Theorem 3.18]{belfi}. It states that, for $\f$ a Riemannian foliation on a compact manifold $M$, if $X\in\mathfrak{L}(\f)$ is $\f$-nondegenerate, then
$$\chi_B(\f)=\sum_{J}\mathrm{ind}_{J}(X)\chi_B(J,\f,\mathrm{Or}_{J}(X)),$$
where the sum ranges over all critical leaf closures $J$ of $\f$. The concepts involved, such as $\f$-nondegeneracy and the index $\mathrm{ind}_{J}(X)$, are generalizations of the classical analogs, and $\chi_B(J,\f,\mathrm{Or}_{J}(X))$ is the alternate sum of the cohomology groups of the complex of $\f|_J$-basic forms with values in the orientation line bundle of $X$ at $J$ (for details, see \cite[Section 3]{belfi}).

\begin{theorem}\label{theo: localization basic euler char}
Let $\f$ be a Riemannian foliation of a closed manifold $M$. If $\overline{X}\in\mathfrak{iso}(\f)$, then $\chi_B(\f)=\chi_B(\f|_{\Zero(\overline{X})})$.
\end{theorem}

\begin{proof}
We are going to construct a suitable vector field $W\in\mathfrak{L}(\f)$ and apply the Basic Hopf Index Theorem. Let $N_1,\dots,N_k$ be the connected components of $\Zero(\overline{X})$. For each $i\in\{1,\dots, k\}$, choose $f_i\in\Omega_B^0(\f|_{N_i})$ a basic Bott--Morse function and $\mathrm{Tub}_\varepsilon(N_i)$ a saturated tubular neighborhood of $N_i$ of radius $\varepsilon>0$ with orthogonal projection $\pi_i:\mathrm{Tub}_\varepsilon(N_i)\to N_i$. Assume $\varepsilon$ sufficiently small so that the tubular neighborhoods are pairwise disjoint.

It is not difficult to see that each $\pi_i$ is a foliate map, so we have that
$$\widetilde{f}_i:=f_i\circ\pi_i:\mathrm{Tub}_{\varepsilon/2}(N_i)\longrightarrow\mathbb{R}$$
is a basic function. Consider $\phi_i$ a basic bump function for $\mathrm{Tub}_{\varepsilon/4}(N_i)$ satisfying $\mathrm{supp}(\phi_i)\subset\mathrm{Tub}_{\varepsilon/2}(N_i)$ and define $Y_i:=\phi_i\grad(\widetilde{f}_i)$. Then $Y_i$ is foliate and if $v$ is a vector in the vertical bundle of $\pi_i$ then
$$\mathrm{g}(Y_i,v)=\mathrm{g}(\phi_i\grad(\widetilde{f}_i),v)=\phi_i\dif(f_i\circ\pi_i)v=\phi_i\dif f_i(\dif \pi_i v)=0,$$
hence $Y_i$ is $\pi_i$-horizontal.

Now let $\{\varphi,\psi\}$ be a basic partition of unity subordinate to the saturated open cover
$$M=\left[\bigsqcup_i\mathrm{Tub}_\varepsilon(N_i)\right]\cup\left[M\Big\backslash\overline{\bigsqcup_i\mathrm{Tub}_{\varepsilon/2}(N_i)}\right]$$
and define $Z_i:=\varphi\grad(d^2_{N_i})$, where $d_{N_i}:M\to\mathbb{R}$ is the distance function to $N_i$. Then clearly $Z_i$ is $\pi_i$-vertical and $\mathrm{supp}(Z_i)\subset\mathrm{Tub}_{\varepsilon/2}(N_i)$. Moreover, one proves that the flow of each $Z_i$ preserves $\f$, hence $Z_i$ is foliate.

The vector field that we are interested in is
$$W:=\psi X+\sum_{i=1}^k(Y_i+Z_i),$$
where $X\in\mathfrak{L}(\f)$ is a fixed representative for $\overline{X}$ (and it is understood that $Y_i$ and $Z_i$ are extended by zero outside their original domains). On $M\setminus\mathrm{Tub}_\varepsilon(\Zero(\overline{X}))$ we have $\overline{W}=\overline{X}\neq 0$. By passing to local quotients we easily see that $X$ is $\pi_i$-horizontal on each $\mathrm{Tub}_\varepsilon(N_i)$, therefore, on $\mathrm{Tub}_\varepsilon(N_i)\setminus\mathrm{Tub}_{\varepsilon/2}(N_i)$, we have $\overline{W}=\psi \overline{X}+\overline{Z}_i\neq0$, because $Z_i$ is $\pi$-vertical. Thus, the critical leaf closures of $W$ must be within $\mathrm{Tub}_{\varepsilon/2}(N_i)$, where we have $W=Y_i+Z_i$. Since $Y_i$ is $\pi$-horizontal and $\Zero(Z_i|_{\mathrm{Tub}_{\varepsilon/2}(N_i)})=N_i$, we conclude that the critical leaf closures of $W|_{\mathrm{Tub}_{\varepsilon}(N_i)}$ coincide with those of $Y_i$ (and are, therefore, within $N_i$).

Let $J=\overline{L}$ be one of those critical leaf closures of $W|_{\mathrm{Tub}_{\varepsilon}(N_i)}$ and let $(x_1,\dots,x_p)$ be normal coordinates on a neighborhood $V\subset L$. Now choose an orthonormal frame
$$(E_1,\dots,E_r,E_{r+1},\dots,E_{\overline{q}},E_{\overline{q}+1},\dots,E_q)$$
for $TL^\perp$ on $V$ such that $(E_1,\dots,E_r)$ is an orthonormal frame for $TJ^\perp\cap TN_i$ and $(E_{r+1},\dots,E_{\overline{q}})$ is an orthonormal frame for $TN_i^\perp$ (in particular, $(E_1,\dots,E_{\overline{q}})$ forms an orthonormal frame for $TJ^\perp$). Applying $\exp^\perp_{L}$ we get coordinates
$$(x_1,\dots,x_p,y_1,\dots,y_r,y_{r+1},\dots,y_{\overline{q}},y_{\overline{q}+1},\dots,y_q)$$
for a tubular neighborhood $\mathrm{Tub}_\delta(V)$ such that $(x_1,\dots,x_p,y_1,\dots,y_r,y_{\overline{q}+1},\dots,y_q)$ are local coordinates for $N_i$. Choose $\delta$ small enough so that $\phi_i|_{\mathrm{Tub}_\delta(V)}\equiv 1\equiv\varphi|_{\mathrm{Tub}_\delta(V)}$. The expression of $W|_\Theta=\grad(\widetilde{f}_i)+\grad(d^2_{N_i})$ in those coordinates is simply
$$W=\sum_{j=1}^r\frac{\partial \widetilde{f}_i}{\partial y_j}\frac{\partial}{\partial y_j}+\sum_{j=r+1}^{\overline{q}}2y_j\frac{\partial}{\partial y_j},$$
so, noting that $\displaystyle\frac{\partial^2 \widetilde{f}_i}{\partial y_j\partial y_k}(x)=\displaystyle\frac{\partial^2 f_i}{\partial y_j\partial y_k}(x)$ for $1\leq j,k\leq r$, we see that, for $x\in J\cap V$, the linear part $W_{J,x}$ of $W$ (see \cite[p. 325]{belfi}) has the matrix representation
$$
\sbox0{$\begin{matrix}
\displaystyle\frac{\partial^2 f_i}{\partial y_1\partial y_1}(x) & \dots & \displaystyle\frac{\partial^2 f_i}{\partial y_1\partial y_r}(x)\\
\vdots & \ddots & \vdots \\
\displaystyle\frac{\partial^2 f_i}{\partial y_r\partial y_1}(x) & \dots & \displaystyle\frac{\partial^2 f_i}{\partial y_r\partial y_r}(x)\vspace{7pt}\end{matrix}$}
\sbox1{$\stackrel{\ }{\begin{matrix}
2& & \\
 & \ddots & \\
 & & 2\end{matrix}}$}
\sbox2{\LARGE $0$}
W_{J,x}\equiv\left[ \begin{array}{c:c}
\usebox{0} & \usebox{2}\\ \hdashline
\usebox{2} & \usebox{1}\end{array}\right]
$$
on $TJ^\perp$. Therefore $W$ is $\f$-nondegenerate.

On the other hand, consider $(N_i,\f|_{N_i})$ and the restriction $W|_{N_i}$. Using the coordinates $(y_1,\dots,y_r,y_{\overline{q}+1},\dots,y_q)$ we see that the matrix representation of $(W|_{N_i})_{J,x}$ coincides with the first block of $W_{J,x}$. In particular, $\ind_{J}(W)=\ind_{J}(W|_{N_i})$. Moreover, this also shows that we can identify $\mathrm{Or}_{J}(W)$ with $\mathrm{Or}_{J}(W|_{N_i})$, since the negative directions of $W_{J,x}$ and $(W|_{N_i})_{J,x}$ coincide. Therefore $\chi_B(J,\f,\mathrm{Or}_{J}(W))=\chi_B(J,\f|_{N_i},\mathrm{Or}_{J}(W|_{N_i}))$.

We now apply \cite[Theorem 3.18]{belfi} to $(M,\f)$ and to each $(N_i,\f|_{N_i})$, obtaining
\begin{eqnarray*}
\chi_B(\f)&=&\sum_{J}\mathrm{ind}_{J}(W)\chi_B(J,\f,\mathrm{Or}_{J}(W))\\
 & = & \sum_i\sum_{J}\mathrm{ind}_{J}(W|_{N_i})\chi_B(J,\f,\mathrm{Or}_{J}(W|_{N_i}))\\
 &=& \sum_i \chi_B(\f|_{N_i}).\end{eqnarray*}
It is clear that basic cohomology is additive under disjoint unions, so the result follows.\end{proof}

\subsection{Proof of Theorem \ref{theorem: basic euler char localizes to closed leaves} and some corollaries}

It is now easy to establish Theorem \ref{theorem: basic euler char localizes to closed leaves}. By Proposition \ref{prop: X killing com Zero=Sigma dim f} we can choose a transverse Killing vector field $\overline{X}\in\mathfrak{iso}(\f)$ such that $\Zero(\overline{X})=\Sigma^{\dim\f}$, thus Theorem \ref{theo: localization basic euler char} gives us $\chi_B(\f)=\chi_B(\f|_{\Sigma^{\dim\f}})$. Now Proposition \ref{prop: basic cohomology of closed foliations} yields
$$\chi_B(\f|_{\Sigma^{\dim\f}})=\chi(\Sigma^{\dim\f}/\!/\f)=\chi(\Sigma^{\dim\f}/\f),$$
where the last equality follows from Theorem \ref{theorem: Satake}.\qed
\vspace{8pt}

Theorem \ref{theorem: basic euler char localizes to closed leaves} generalizes \cite[Corollary 1]{gorokhovsky}, where this result is obtained from the study of the index of transverse operators under additional assumptions on $\f$.

\begin{corollary}\label{corollary: if no closed leaves then euler char vanishes}
If a Killing foliation $\f$ of a compact manifold has no closed leaves, then $\chi_B(\f)=0$.
\end{corollary}

The theorem by P.~Conner \cite{conner} mentioned in the beginning of this chapter is originally stated for the set of fixed points of a torus action instead of $\Zero(X)$ (the statement for $\Zero(X)$ can be recovered by considering the action of the closure in $\mathrm{Iso}(M)$ of the subgroup generated by the flow of $X$). In analogy, if $\f$ is a Killing foliation, the stratum of closed leaves $\Sigma^{\dim\f}$ corresponds to the fixed point set of the transverse action of the Abelian algebra $\mathfrak{g}_\f$. The following result can be seen, thus, as a version of Conner's Theorem for foliations.

\begin{corollary}\label{corollary: Conner via goertsches}
Let $\f$ be a transversely oriented Killing foliation of a compact manifold $M$. Then
\begin{enumerate}[(i)]
\item $\displaystyle \sum_i b^{2i}_B(\f)\geq \sum_i b^{2i}_B(\f|_{\Sigma^{\dim\f}})$,
\item $\displaystyle \sum_i b^{2i+1}_B(\f)\geq \sum_i b^{2i+1}_B(\f|_{\Sigma^{\dim\f}})$.
\end{enumerate}
\end{corollary}

\begin{proof}
It is shown in \cite[Theorem 5.5]{goertsches} that
\begin{equation}\label{eq: toben conner}\sum_i b^{i}_B(\f)\geq\sum_i b^i_B(\f|_{\Sigma^{\dim\f}}).\end{equation}
On the other hand, $\chi_B(\f)=\chi(\Sigma^{\dim\f}/\f)$ gives
\begin{equation}\label{eq: corolary expanded} \sum_i (-1)^ib^{i}_B(\f)=\sum_i (-1)^ib^{i}_B(\f|_{\Sigma^{\dim\f}}).\end{equation}
By adding \eqref{eq: toben conner} and \eqref{eq: corolary expanded} we get the first item and by subtracting \eqref{eq: corolary expanded} from \eqref{eq: toben conner} we get the second one.
\end{proof}

\subsection{The basic Euler characteristic under deformations}

We will show now that the basic Euler characteristic is preserved by the deformation given in Theorem \ref{theorem: Haefliger deformation}. This is somewhat surprising, because the basic Betti numbers, in general, are not preserved.

\begin{theorem}\label{theorem: basic euler char is preserved by deformations}
Let $\f$ be a Killing foliation of a compact manifold $M$ and let $\g$ be obtained from $\f$ by a deformation as in Theorem \ref{theorem: Haefliger deformation}. Then
$$\chi_B(\f)=\chi_B(\g).$$
\end{theorem}

\begin{proof}
From Theorem \ref{theorem: Haefliger deformation} we have that $M/\overline{\f}=(M/\g)/\mathbb{T}^d$. In particular, the points in $(M/\g)^{\mathbb{T}^d}$ correspond to the closed leaves of $\f$, that is $(M/\g)^{\mathbb{T}^d}=\Sigma^{\dim\f}/\f$. Hence
$$\chi_B(\g)=\chi(M/\g)=\chi\left((M/\g)^{\mathbb{T}^d}\right)=\chi(\Sigma^{\dim\f}/\f) = \chi_B(\f),$$
where the second equality follows from the theory of continuous torus actions \cite[Theorem 10.9]{bredon} and the last one from Theorem \ref{theo: localization basic euler char}.
\end{proof}

\section{Transverse Hopf conjecture}

A well-know conjecture by H.~Hopf states that any even-dimensional, compact Riemannian manifold $M$ with positive sectional curvature must have positive Euler characteristic. In dimension $2$ the conjecture holds, for if $\ric_M$ is bounded below by a positive constant we have $|\pi(M)|<\infty$, hence $H_1(M,\mathbb{R})=0$ by the Hurewicz Theorem. In dimension $4$, passing to the orientable double cover and using Poincaré duality we see that the conjecture also holds. For dimensions larger than $4$ the full conjecture is still an open problem, but it holds, however, when the symmetry rank of $M$ is sufficiently large. For example, if $\dim(M)=6$ and $M$ admits a non-trivial Killing vector field $X$, then, by the lower dimensional cases, $0<\chi(\Zero(X))=\chi(M)$. For dimensions larger than $6$, there is the following result by T.~Püttmann and C.~Searle \cite[Theorem 2]{puttmann}.

\begin{theorem}[Püttmann--Searle]\label{theorem: Puttmann-Searle classic} Let $M$ be an $n$-dimensional complete Riemannian manifold such that $\sec_M\geq c>0$. If $n$ is even and $\symrank(M)\geq n/4-1$, then $\chi(M)>0$.\end{theorem}

Our goal in this section is to obtain a transverse version of this result for Killing foliations. Let us begin by trying to directly generalize the low codimensional cases from the manifold counterparts that we saw above. Suppose that $\f$ is a $2$-codimensional complete Riemannian foliation on $M$ and that $\ric_\f\geq c>0$. By Theorem \ref{theorem: hebda} if follows that $H_B^1(\f)=0$. Furthermore, since $\Omega_B^0(\f)$ consists of the basic functions, which are precisely the functions $f:M\to\mathbb{R}$ that are constant on the closures of the leaves, we have $H_B^0(\f)= H_{\mathrm{dR}}^0(M)$, thus
$$\chi_B(\f)=b_B^0(\f)+b_B^2(\f)=\dim H_{\mathrm{dR}}^0(M)+b_B^2(\f)\geq 1,$$
which establishes the following.

\begin{proposition}\label{proposition: chi positive for codim 2}
Let $\f$ be a $2$-codimensional complete Riemannian foliation satisfying $\ric_\f\geq c>0$. Then $\chi_B(\f)>0$.
\end{proposition}

In order to prove the foliation version of Hopf's conjecture to $4$-codimensional foliations we need Poincaré duality to hold for the basic cohomology complex. Combining Theorem \ref{theorem: hebda} with the results in \cite{lopez} we see that $\ric_\f>0$ is a sufficient condition for this to happen, provided the foliation is also transversely oriented \cite[Corollary 6.2, Theorem 6.4 and Corollary 6.5]{lopez}.

\begin{lemma}[López]\label{lemma: poincare duality for positive ricci}
Let $\f$ be a transversely oriented complete Riemannian foliation of a compact manifold such that $\ric_\f>0$. Then $H^i_B(\f)\cong H^{q-i}_B(\f)$.
\end{lemma}

\begin{proposition}\label{proposition: chi positive for codim 4}
Let $\f$ be a $4$-codimensional, transversely oriented Riemannian foliation of a compact manifold $M$ satisfying $\ric_\f>0$. Then $\chi_B(\f)> 0$.
\end{proposition}

\begin{proof}
Since $H_B^1(\f)=0$, by Theorem \ref{theorem: hebda}, and $H_B^1(\f)\cong H_B^3(\f)$ by Lemma \ref{lemma: poincare duality for positive ricci}, we have $\chi_B(\f)=b_B^0(\f)+b_B^2(\f)+b_B^4(\f)\geq 2$.
\end{proof}

\subsection{Foliations admitting a Killing vector field with a large zero set}

The Grove--Searle classification of manifolds with maximal symmetry rank and positive sectional curvature is achieved in \cite{grove} by reducing it to the classification of positively curved manifolds admitting a Killing vector field $X$ such that $\Zero(X)$ has a connected component $N$ with $\codim(N)=2$ \cite[Theorem 1.2]{grove}. In this section we will study, analogously, the leaf spaces of closed Riemannian foliations that admit a transverse Killing vector field $\overline{X}$ such that $\codim(N)=2$ for some connected component $N$ of $\Zero(\overline{X})$. We do so because it will be useful later in our pursuit of a transverse version of Theorem \ref{theorem: Puttmann-Searle classic}. We begin with the following transverse version of Frankel's Theorem.

\begin{lemma}\label{lemma: Fraenkel for foliations}
Let $(M,\f)$ be a complete Riemannian foliation with $\sec_\f\geq c>0$ and let $N$ and $N'$ be $\f$-saturated, horizontally totally geodesic, compact submanifolds such that $\codim(\f|_{N})+\codim(\f|_{N'})\geq\codim(\f)$. Then $N\cap N'\neq \emptyset$.
\end{lemma}

The proof is similar to that of the classical case \cite[Theorem 1]{fraenkel} if one works with $\mathscr{H}_\f$-geodesics on a total transversal $T_\f$, so we will omit it.

\begin{theorem}\label{theorem: quotient when there is a codimension 2 zero set}
Let $\f$ be a $q$-codimensional, closed Riemannian foliation of a closed manifold $M$. Suppose that $q$ is even, $\sec_\f c>0$ and $\overline{X}\in\mathfrak{iso}(\f)$ satisfies $\codim(N)=2$ for some connected component $N$ of $\Zero(\overline{X})$. Then $M/\f$ is homeomorphic to the quotient space of either $\mathbb{S}^q$ or $|\mathbb{CP}^{q/2}[\lambda]|$ by the action of a finite group.
\end{theorem}

\begin{proof}
Denote $\mathcal{O}:=M/\!/\f$. Then $\overline{X}$ induces $\overline{X}_\mathcal{O}\in\mathfrak{iso}(\mathcal{O})$. Let $T$ be the closure of the subgroup generated by the flow of $\overline{X}_\mathcal{O}$ in $\mathrm{Iso}(\mathcal{O})$. It is clear that $T$ is a torus and that $\overline{N}:=\pi(N)$ is a connected component of the fixed-point set $\mathcal{O}^T$. Choose a closed $1$-parameter subgroup $\mathbb{S}^1<T$ such that $\mathcal{O}^{\mathbb{S}^1}=\mathcal{O}^T$ \cite[Lemma 4.2.1]{allday}. Then $|\mathcal{O}|$, with the distance function induced by $\mathrm{g}^T$, is a positively curved Alexandrov space admitting fixed-point homogeneous action of $\mathbb{S}^1$, in the terminology of \cite[Section 6]{harvey}. By \cite[Theorem 6.5]{harvey}, there is a unique orbit $\mathbb{S}^1x$ at maximal distance from $\overline{N}$, the ``soul'' orbit, and there is an $\mathbb{S}^1$-equivariant homeomorphism
$$|\mathcal{O}|\cong \frac{\nu_x*\mathbb{S}^1}{\mathbb{S}^1_x},$$
where $*$ denotes the join operation, $\nu_x$ is the space of normal directions to $\mathbb{S}^1x$ at $x$ and $\mathbb{S}^1_x$ acts on the left on $\nu_x*\mathbb{S}^1$, the action on $\nu_x$ being the isotropy action and the action on $\mathbb{S}^1$ being the inverse action on the right. Notice that $\nu_x$ can be identified with $\mathbb{S}^{\codim(\mathbb{S}^1x)-1}/\Gamma_x$, where $\mathbb{S}^{\codim(\mathbb{S}^1x)-1}$ is the unit sphere in $T_x(\mathbb{S}^1x)^\perp\subset T_x\mathcal{O}$ and $\Gamma_x$ the local group of $\mathcal{O}$ at $x$. By \cite[Proposition 2.12 and Corollary 2.13]{galazgarcia} we can choose an orbifold chart $(\widetilde{U},\Gamma_x,\phi)$ and an extension
$$0\longrightarrow\Gamma_x\longrightarrow\widetilde{\mathbb{S}}^1_x\stackrel{\rho}{\longrightarrow} \mathbb{S}^1_x\to 0$$
acting on $\widetilde{U}$ with $\widetilde{U}/\widetilde{\mathbb{S}}^1_x=U/\mathbb{S}^1_x$ (let us denote this action by $\mu$). We now consider separately the cases when $\mathbb{S}^1_x$ is a finite cyclic group $\mathbb{Z}_r$ and when $\mathbb{S}^1_x=\mathbb{S}^1$.

Suppose that $\mathbb{S}^1_x\cong \mathbb{Z}_r$. Then $\dim(\mathbb{S}^1x)=1$, hence $\nu_x\cong \mathbb{S}^{q-2}/\Gamma_x$, and $\widetilde{\mathbb{S}}^1_x$ is finite. Recall that there is an isometry $\mathbb{S}^{m}*\mathbb{S}^{n}\cong \mathbb{S}^{m+n+1}$ when we realize $\mathbb{S}^{m}*\mathbb{S}^{n}$ via the map
$$
\begin{array}{rcl}
\mathbb{S}^m\times\mathbb{S}^n\times [0,1]& \longrightarrow &\mathbb{S}^{m+n+1}\subset\mathbb{R}^{m+n+2}\\
(s_1,s_2,t) & \longmapsto & \left(\cos\left(\frac{\pi}{2}t\right)s_1,\sin\left(\frac{\pi}{2}t\right)s_2\right).
\end{array}$$
If we define an isometric action of $\widetilde{\mathbb{S}}^1_x$ on $\mathbb{S}^q\cong\mathbb{S}^{q-2}*\mathbb{S}^1$ via this map by
$$\left(g,\left(\cos\left(\frac{\pi}{2}t\right)s_1,\sin\left(\frac{\pi}{2}t\right)s_2\right)\right)\longmapsto \left(\cos\left(\frac{\pi}{2}t\right)\dif(\mu^g)_{\widetilde{x}}s_1,\sin\left(\frac{\pi}{2}t\right)s_2\rho(g)^{-1}\right),$$
we get
$$|\mathcal{O}|\cong \frac{\mathbb{S}^{q-2}/\Gamma_x*\mathbb{S}^1}{\mathbb{Z}_r}\cong\frac{\mathbb{S}^q}{\widetilde{\mathbb{S}}^1_x},$$
which exhibits $|\mathcal{O}|$ as a finite quotient of a sphere.

Now suppose $\mathbb{S}^1_x\cong \mathbb{S}^1$. Then $\dim(\mathbb{S}^1x)=0$ and $\nu_x\cong \mathbb{S}^{q-1}/\Gamma_x$, so, similarly, we have
$$|\mathcal{O}|\cong \frac{\mathbb{S}^{q-1}/\Gamma_x*\mathbb{S}^1}{\mathbb{S}^1}\cong \frac{\mathbb{S}^{q-1}*\mathbb{S}^1}{\widetilde{\mathbb{S}}^1_x}\cong \frac{\mathbb{S}^{q+1}}{\widetilde{\mathbb{S}}^1_x}.$$
In this case $x$ is a fixed point of the $\mathbb{S}^1$ action and therefore corresponds to a leaf $L$ of $\f$ where $\overline{X}$ vanishes. We claim that $x$ is an isolated fixed point. Indeed, if this was not the case, since $q=\codim(\f)$ is even, the connected component $N_L$ of $\Zero(\overline{X})$ containing $L$ would satisfy $\codim(N_L)\leq q-2$ (see Proposition \ref{proposition: propriedade zero killing}), hence
$$\codim(\f|_N)+\codim(\f|_{N_L})=2q-\codim(N)-\codim(N_L)\geq q,$$
so, by Lemma \ref{lemma: Fraenkel for foliations}, $N_L\cap N\neq\emptyset$, which translates to $x\in \overline{N}$, absurd. It follows, therefore, that $\widetilde{\mathbb{S}}^1_x$ acts almost freely on the first join factor $\mathbb{S}^{q-1}$ that corresponds to the space of normal directions to $\mathbb{S}^1x$ and, hence, that its induced action on $\mathbb{S}^{q+1}$ is also almost free.

Let $\mathcal{E}$ be the Riemannian foliation of $\mathbb{S}^{q+1}$ given by the connected components of the orbits of $\widetilde{\mathbb{S}}^1_x$. Notice that $\widetilde{\mathbb{S}}^1_x$ may be disconnected, but the connected component $(\widetilde{\mathbb{S}}^1_x)_0$ of the identity is a circle whose action defines the same foliation $\mathcal{E}$. Thus, denoting by $\Lambda$ the finite group $\widetilde{\mathbb{S}}^1_x/(\widetilde{\mathbb{S}}^1_x)_0$, it follows that
$$\frac{\mathbb{S}^{q+1}}{\widetilde{\mathbb{S}}^1_x}\cong\frac{\mathbb{S}^{q+1}/(\widetilde{\mathbb{S}}^1_x)_0}{\Lambda}=\frac{\mathbb{S}^{q+1}/\mathcal{E}}{\Lambda},$$
where the action of $\Lambda$ identifies the points in $\mathbb{S}^{q+1}/\mathcal{E}$ corresponding to the same $\widetilde{\mathbb{S}}^1_x$-orbit. In view of the classification of Riemannian $1$-foliations of the sphere \cite[Theorem 5.4]{gromoll2}, and since $\mathcal{E}$ is closed, we obtain $|\mathcal{O}|\cong |\mathbb{CP}^{n/2}[\lambda]|/\Lambda$.
\end{proof}

\subsection{Püttmann--Searle Theorem for orbifolds}

Recall that we obtained a transverse version of Conner's Theorem in Corollary \ref{corollary: Conner via goertsches}. Notice, however, that a complete analog of Conner's theorem should be stated for the zero set of any transverse Killing vector field, in place of $\Sigma^{\dim\f}$. Unfortunately, the result in \cite{goertsches} that we use in the proof of Corollary \ref{corollary: Conner via goertsches} cannot be adapted to show this for a general Killing foliation. The following version for closed foliations, however, will be useful.

\begin{proposition}\label{proposition: Conner for closed foliations}
Let $\f$ be a closed Riemannian foliation of a closed manifold $M$ and let $\overline{X}\in\mathfrak{iso}(\f)$. Then
$$\sum_i b^{2i+k}_B(\f)\geq \sum_i b^{2i+k}_B(\f|_{\Zero(\overline{X})}),$$
for $k=0,1$.
\end{proposition}

\begin{proof}
Let us denote $\mathcal{O}:=M/\!/\f$. By Proposition \ref{prop: basic cohomology of closed foliations} and Theorem \ref{theorem: Satake} we have
$$b^i_B(\f)=\dim(H_{\mathrm{dR}}^i(\mathcal{O}))=\dim(H^i(|\mathcal{O}|,\mathbb{R})).$$
Furthermore, since $|\mathcal{O}|$ is paracompact and locally contractible, its singular cohomology coincides with its \v{C}ech cohomology, so we also have
$$b^i_B(\f)=\rank (\check{H}^i(|\mathcal{O}|,\mathbb{R})).$$

Now consider the closure of the subgroup generated by the flow of the induced Killing vector field $\overline{X}_{\mathcal{O}}\in\mathfrak{iso}(\mathcal{O})$. This subgroup is a torus $T<\mathrm{Iso}(\mathcal{O})$ and $\Zero(\overline{X}_{\mathcal{O}})=|\mathcal{O}|^T$, the fixed-point set of its action. From the theory of continuous torus actions \cite[Theorem 10.9]{bredon} it follows that
$$\sum_i \rank (\check{H}^{2i+k}(|\mathcal{O}|,\mathbb{R}))\geq \sum_i \rank (\check{H}^{2i+k}(|\mathcal{O}|^T,\mathbb{R})).\qedhere$$
\end{proof}

The lemma below is a transverse version of Theorem \ref{theorem: Puttmann-Searle classic} for closed foliations, from which the orbifold analogue of the Püttmann--Searle Theorem will be a direct consequence.

\begin{lemma}\label{lemma: putmann searle for closed foliations}
Let $\f$ be a $q$-codimensional, transversely orientable, closed Riemannian foliation of a closed manifold $M$ and let $N\in\mathcal{Z}(\mathfrak{a})$, where $\mathfrak{a}<\mathfrak{iso}(\f)$ is any Abelian Lie subalgebra such that $\dim(\mathfrak{a})=\symrank(\f)$. If $q$ is even, $\sec_\f>0$ and $\symrank(\f)\geq q/4-1$, then $\chi_B(\f|_N)>0$. In particular, $\chi_B(\f)>0$.
\end{lemma}

\begin{proof}
We proceed by induction on $q$. Notice that $\codim(\f|_N)=\codim(\f)-\codim(N)$ is always even and $\sec_{\f|_N}>0$, by Proposition \ref{proposition: propriedade zero killing}. For $q< 6$ the result follows directly from Propositions \ref{proposition: chi positive for codim 2} and \ref{proposition: chi positive for codim 4}.

For the induction step, take a maximal $N'\in\mathcal{Z}(\mathfrak{a})$ such that $N\subset N'$. We now have two cases.

If $\codim(N')\geq 4$, then $\codim(\f|_{N'})\leq\codim(\f)-4$. Therefore, by Proposition \ref{proposition: propriedades de Zero killing},
\begin{eqnarray*}
\symrank(\f|_{N'}) & = & \dim(\mathfrak{a}|_{N'})=\dim(\mathfrak{a})-1=\symrank(\f)-1\\
           & \geq & \frac{\codim(\f)-4}{4}-1\geq\frac{\codim(\f|_{N'})}{4}-1,\end{eqnarray*}
so $(N',\f|_{N'})$ satisfy the induction hypothesis and $\chi_B(\f|_N)>0$, because $N\in\mathcal{Z}(\mathfrak{a}|_{N'})$.

If $\codim(N')= 2$, then using Theorem \ref{theorem: quotient when there is a codimension 2 zero set} we have that
$$M/\f\cong\begin{cases}
|\mathbb{CP}^{\frac{q}{2}}[\lambda]|/\Lambda \text{ or}\\
\mathbb{S}^q/\Lambda,
\end{cases}$$
where $\Lambda$ is a finite group in either case. Now, by \cite[Theorem 7.2]{bredon},
$$\check{H}^*(M/\f,\mathbb{R})\cong\begin{cases}
\check{H}^*(\mathbb{CP}^{\frac{q}{2}}[\lambda],\mathbb{R})^\Lambda \text{ or}\\
\check{H}^*(\mathbb{S}^q,\mathbb{R})^\Lambda.
\end{cases}$$
Since all odd Betti numbers of both $\mathbb{CP}^{q/2}[\lambda]$ and $\mathbb{S}^q$ vanish, we obtain $b_{2i+1}(M/\f)=b^{2i+1}_B(\f)=0$ for all $i>0$. From Proposition \ref{proposition: Conner for closed foliations},
$$0\geq \sum_i b^{2i+1}_B(\f|_{\Zero(\overline{X})})\geq \sum_i b^{2i+1}_B(\f|_N),$$
so $b^{2i+1}_B(\f|_N)$ also vanishes for all $i$. In particular, $\chi_B(\f|_N)>0$.
\end{proof}

\begin{corollary}[Püttmann--Searle Theorem for orbifolds]\label{corollary: Putmann for orbifolds}
Let $\mathcal{O}$ be a compact, orientable $n$-dimensional Riemannian orbifold such that $\sec_{\mathcal{O}}>0$. If $n$ is even and $\symrank(\mathcal{O})\geq n/4-1$, then $\chi(|\mathcal{O}|)>0$.
\end{corollary}

\begin{proof}
By \cite[Proposition 5.21]{alex}, $\mathcal{O}$ is the leaf space of a closed, transversely orientable, $n$-codimensional Riemannian foliation $\f$ of a closed manifold, satisfying $\sec_\f =\sec_{\mathcal{O}}>0$. Moreover, the identification $\mathfrak{iso}(\mathcal{O})\cong\mathfrak{iso}(\f)$ gives us that $\symrank(\f)\geq n/4-1$, therefore $\f$ satisfies the hypotheses of Lemma \ref{lemma: putmann searle for closed foliations}. From Proposition \ref{prop: basic cohomology of closed foliations} and Theorem \ref{theorem: Satake} it follows that
$$\chi(|\mathcal{O}|)=\sum_i(-1)^i\dim(H_{\mathrm{dR}}^i(\mathcal{O}))=\chi_B(\f)>0.\qedhere$$
\end{proof}

\subsection{Proof of Theorem \ref{theorem: puttmann for foliations}}

Since $M$ is compact we actually have $\sec_\f\geq c>0$ for some constant $c$. Therefore, we can use Theorem \ref{theorem: Haefliger deformation} to deform $\f$ into a closed Riemannian foliation $\g$ such that $\sec_\g>0$ and $\symrank(\g)\geq d=\dim(\overline{\f})-\dim(\f)$. By Lemma \ref{lemma: putmann searle for closed foliations} and Theorem \ref{theorem: basic euler char is preserved by deformations} we have
$$\chi_B(\f)=\chi_B(\g)>0.$$

\section{A topological obstruction}\label{section: topological obstruction}

Let $(M,\f)$ be a Riemannian foliation and denote by $\mathrm{G}_\f$ the \textit{holonomy groupoid} of $\f$ (see \cite[Proposition 5.6]{mrcun} or \cite[Section 1.1]{haefliger3}). The study of classifying spaces of topological groupoids that appears in \cite{haefliger3} shows that there is a locally trivial fibration $E\mathrm{G}_\f\to B\mathrm{G}_\f$, whose fiber is a generic leaf of $\f$ \cite[Corollaire 3.1.5]{haefliger3}, and a commutative diagram
$$\xymatrix{
E\mathrm{G}_\f \ar[dd] \ar[dr]^-{\zeta} & \\
 & M \ar[dl]^-{\Upsilon}\\
B\mathrm{G}_\f &}$$
where $\zeta$ is a homotopy equivalence \cite[Coloraire 3.1.4]{haefliger3}. Analogously, this construction can be applied to the groupoid $\mathrm{G}^T_\f$ of a representative $(T_\f,\mathscr{H}_\f)$ of the holonomy pseudogroup of $\f$ and, since $\mathrm{G}_\f$ and $\mathrm{G}^T_\f$ are equivalent \cite[p. 81]{haefliger3}, it follows from \cite[Corollaire 3.1.3]{haefliger3} that $B\mathrm{G}_\f$ and $B\mathrm{G}^T_\f$ are homotopy equivalent. Moreover, notice that when $\f$ is a closed foliation $B\mathrm{G}^T_\f$ coincides with the classifying space of the orbifold $M/\!/\f$, so the above results show that, at least from the homotopy theoretic point of view, closed foliations behave essentially like fibrations. We can now combine these facts with the invariance of the basic Euler characteristic under deformations (see Theorem \ref{theorem: basic euler char is preserved by deformations}) to obtain the following.

\begin{theorem}\label{theorem: closed leaf + transverse symmetry implies charM vanishes}
Let $\f$ be a Riemannian foliation of a closed, simply-connected manifold $M$. If $\chi(M)\neq 0$ then $\f$ is closed.
\end{theorem}

\begin{proof} Assume that $\chi(M)\neq 0$ and that $\f$ is not closed. By \cite[Théorème 3.5]{ghys}, there is a closed leaf $L\in\f$. If we choose a closed foliation $\g$ near $\f$ via Theorem \ref{theorem: Haefliger deformation}, then $L$ is also a leaf of $\g$, because the deformation respects $\overline{\f}$. Fix a generic leaf $\widehat{L}\in\g$ near $L$. By the results in \cite{haefliger3} presented above, we have a locally trivial fibration $E\mathrm{G}_\g\to B\mathrm{G}_\g$ with fiber $\widehat{L}$ and a commutative diagram
$$\xymatrix{
 & E\mathrm{G}_\g \ar[dd] \ar[dr]^-{\zeta} & \\
 &  & M \ar[dl]^-{\Upsilon}\\
B\mathcal{O}\ar[r]^-{h} & B\mathrm{G}_\g &}$$
where $\mathcal{O}$ denotes $M/\!/\g$ and both $\zeta$ and $h$ are homotopy equivalences. By the homotopy exact sequence of $E\mathrm{G}_\g\to B\mathrm{G}_\g$ we see that $B\mathrm{G}_\g$ is also simply-connected, hence the Euler characteristic (with real coefficients) has the product property
$$\chi(M)=\chi(E\mathrm{G}_\g)=\chi(\widehat{L})\chi(B\mathrm{G}_\g).$$
Moreover, it is known that $H^*(B\mathcal{O},\mathbb{R})\cong H^*(|\mathcal{O}|,\mathbb{R})$ \cite[Corollary 4.3.8]{boyer}, so
\begin{equation}\chi(M)=\chi(\widehat{L})\chi(|\mathcal{O}|)=\chi(\widehat{L})\chi_B(\g)=\chi(\widehat{L})\chi_B(\f),\label{eq: fibration property basic euler char}\end{equation}
where we use Theorem \ref{theorem: basic euler char is preserved by deformations} and Proposition \ref{prop: basic cohomology of closed foliations}.

Local Reeb stability asserts that the restriction of the orthogonal projection $\mathrm{Tub}_\varepsilon(L)\to L$ to $\widehat{L}$ is a $h(\g)$-sheeted covering map $\widehat{L}\to L$, where $h(\g)=|\mathrm{Hol}(L)|<\infty$ (the holonomy being with respect to $\g$). Hence $\chi(\widehat{L})=h(\g)\chi(L)$ and equation \eqref{eq: fibration property basic euler char} can be rewritten as
\begin{equation}\chi(M)=h(\g)\chi(L)\chi_B(\f).\label{eq: fibration property basic euler char 2}\end{equation}
On the other hand, by Lemma \ref{lemma: variando a holonomia}, for a sequence $\g_i$ of closed foliations approaching $\f$ we must have $h(\g_{i})\to\infty$. In particular, we can change the number $h(\g)$ by varying $\g$. This violates equation \eqref{eq: fibration property basic euler char 2}.
\end{proof}

\subsection{Proof of Theorem \ref{teo: topological obstruction}}

Suppose $M$ is a compact manifold satisfying $|\pi_1(M)|<\infty$ and $\chi(M)\neq0$ and let $\f$ be a Riemannian foliation of $M$. We denote by $\widehat{\f}$ the lift of $\f$ to the universal covering $\widehat{M}$. Then we have $\chi(\widehat{M})=|\pi_1(M)|\chi(M)\neq0$, so $\widehat{\f}$ is a closed foliation, by Theorem \ref{theorem: closed leaf + transverse symmetry implies charM vanishes}. It now follows from Corollary \ref{corollary: algebra estrutural e levantamentos finitos} that $\f$ is also closed.\qed

\end{document}